\documentclass[reqno, 12pt]{amsart}

\pdfoutput=1

\usepackage{enumerate}
\usepackage{latexsym}
\usepackage[centertags]{amsmath}
\usepackage{amsfonts}
\usepackage{amssymb}
\usepackage{amsthm}
\usepackage{newlfont}
\usepackage{graphics}
\usepackage{color}
\usepackage{float}
\usepackage{diagbox}
\usepackage{longtable}
\usepackage{rotating}
\usepackage{multirow}

\usepackage{extarrows}

\usepackage[sort,compress,numbers]{natbib}

\usepackage[utf8]{inputenc}

\newtheorem{thm}{Theorem}[section]
\newtheorem{cor}[thm]{Corollary}
\newtheorem{lem}[thm]{Lemma}
\newtheorem{prop}[thm]{Proposition}

\theoremstyle{mydefinition}

\theoremstyle{myremark}
\newtheorem{rem}[thm]{Remark}
\newtheorem{exa}[thm]{Example}

\allowdisplaybreaks[4]

\def\N{\mathbb{N}}

\def\CT{\mathop{\mathrm{CT}}}

\title[Frobenius Numbers of Special Sequences]{A Combinatorial Approach to Frobenius Numbers of Some Special Sequences (Complete Version)}

\author{ Feihu Liu$^{1}$ and Guoce Xin$^{2, *}$}

\address{ $^{1, 2}$School of Mathematical Sciences,  Capital Normal University,
 Beijing 100048,  PR China}
\email{$^1$\texttt{liufeihu7476@163.com}\ \  \& $^2$\texttt{guoce\_xin@163.com}}
\date{ March 13,  2023}
\thanks{$*$ This work was partially supported by NSFC(12071311).}
\begin{document}
\maketitle

\begin{abstract}
Let $A=(a_1, a_2, ..., a_n)$ be relative prime positive integers with $a_i\geq 2$. The Frobenius number $g(A)$ is the greatest integer not belonging to the set $\big\{ \sum_{i=1}^na_ix_i\ |x_i\in \mathbb{N}\big\}$. The general Frobenius problem includes the determination of $g(A)$ and the related Sylvester number $n(A)$ and Sylvester sum $s(A)$. We present a new approach to the Frobenius problem.
Basically, we transform the problem into an easier optimization problem. If the new problem can be solved explicitly, then we will be able to obtain a formula of $g(A)$. We illustrate the idea by giving concise proof of some existing formulas and finding some interesting new formulas of $g(A), n(A), s(A)$. Moreover, we find that MacMahon's partition analysis applies to give a new way of calculating $n(A), s(A)$
by using a rational function representation of a polynomial determined by $A$.
\end{abstract}

\def\D{{\mathcal{D}}}

\noindent
\begin{small}
 \emph{Mathematic subject classification}: Primary 11D07; Secondary 05A15,  11B75.
\end{small}

\noindent
\begin{small}
\end{small}

\section{Introduction}
Throughout the paper,  we shall always use the following standard notation: $\mathbb{Z}$ is the set of all integers,  $\mathbb{N}$ is the set of all non-negative integers,  and $\mathbb{P}$ is the set of all positive integers.

Let $A=(a_1, a_2, ..., a_n)$ be a positive integer vector. Sylvester \cite{J. J. Sylvester} defined the function $d(a_0;A)$,  called the denumerant of $a_0\in \N$ with respect to $A$,  by
$$d(a_0;A)=d(a_0;a_1, a_2, ..., a_n)=\#\Big\{ (x_1, ..., x_n)\in \mathbb{N}^n \mid \sum_{i=1}^n a_ix_i=a_0\Big\}.$$
That is,  $d(a_0;A)$ is the number of non-negative integer representations of $a_0$ by $a_1, a_2, ..., a_n$.
We say $a_0$ is \emph{representable} by $A$ if $d(a_0;A)>0$ and \emph{not representable} by $A$ if otherwise.
Denote by $\mathcal{R}=\mathcal{R}(A)$ the set of numbers representable by $A$, and by $\mathcal{NR}=\mathcal{NR}(A)$ the set of numbers not representable by  $A$.

When $\gcd(A)=1$, it is not hard to see that $\mathcal{NR}(A)$ is finite, i.e., any sufficiently large $a_0$ is representable by $A$. Then a natural question is: what is
$$g(A)=\max \mathcal{NR}(A) = \max \{a_0 \in \mathbb{N} \mid d(a_0;A)=0 \},$$
i.e., the maximum number not representable by $A$? This number is now called the Frobenius number of $A$ since it was first studied by Frobenius.
The Frobenius problem is also known as the Coin Exchange Problem. It refers to two
problems: compute the Frobenius number $g(A)$; for a given $m$,  determine if $m$ is representable by $A$,  and find a representation if possible.

The determination of $g(A)$ turns out to be a very hard problem even when $n$ is fixed.
For $n=2$,  it is well known that $g(a_1, a_2)=a_1a_2-a_1-a_2$. This case was studied as early as in 1882 by Sylvester \cite{J. J. Sylvester1}.
For $n=3$, a formula related to rational function appears in \cite{G.Denham}. Amitabha Tripathi also discussed it in \cite{A. Tripathi1}. For $n=4$, no general formula for $g(A)$ is known,  but many formulas for special cases have been determined.

Algorithmically, the determination of $g(A)$ was shown by Ramírez Alfonsín \cite{J. L. Ramrez Alfonsn} to be NP-hard under Turing reduction. Thus
an algorithm for the fixed $n$ case is valuable. A theoretical solution was given by Kannan \cite{R. Kannan}, who use the concept of covering radius
to gave a polynomial time algorithm for $g(A)$ when $n$ is fixed. But the algorithm is impractical due to its high complexity. Many algorithms have been developed
for the small $n$ case. For $n=3$, H. Greenberg provides a fast algorithm with runtime $O(\log a_1)$ in \cite{Greenberg88}. D. Einstein, D. Lichtblau, A. Strzebonski and S. Wagon \cite{D. Einstein} gave a fast (but not polynomial) algorithm for computing the Frobenius number $g(A)$ for up to $n=11$ parameters.

We start to study the Frobenius problem because of the following known result lie in the conjunction of Algebraic Combinatorics and Computational Geometry.
The denumerant $d(a_0;A)$ can be computed in polynomial time when $n$ is fixed. This problem can be solved by: i) Barvinok's algorithm for counting lattice points in a rational convex polytope, with a practical implementation by LattE in \cite{DeLoera04}; ii) the constant term method in \cite{Xin15}, where a polynomial time algorithm was described and a simple (but not polynomial time) implementation by CTEuclid. We believe that the tools developed there can also be applied to the quite related $g(A)$.

The method we are going to present arises from trying to give simple proofs of some known results. For example: Hujter \cite{M. Hujter} gave the formula $g(A)$ for $A=(a, a^2+1, a^2+a), a>2$; Dulmage and Mendelsohn \cite{A. L. Dulmage} obtained some formulas for $A=(a, a+1, a+2, a+4)$,  $A=(a, a+1, a+2, a+5)$, $A=(a, a+1, a+2, a+6)$; For $A=(a, ha+d, ha+bd, ha+b^2d, ..., ha+b^kd)$ is discussed in \cite{A. Tripathi3} by A. Tripathi. For more formulas,  see \cite{Ramrez Alfonsn}.

The case
$A=(a, a+1, ..., a+k)$ is our break point.  Brauer \cite{A. Brauer} first computed $g(A)$ for this case,  Roberts \cite{Roberts1} extended this result to numbers in arithmetic progression. For $g(A)$ and $n(A)$, Selmer \cite{E. S. Selmer} further generalized to the case  $A=(a, ha+d, ha+2d, ..., ha+kd), \ (a, d)=1$.
Our proof (of Theorem \ref{lsf}) is essentially equivalent to that in \cite{A. Tripathi2}. But our treatment is neat, and easy to generalize. This leads to a method
for Forbenius problems.

We follow some notations in \cite{D. Einstein}. By the fundamental result in Lemma  \ref{LiuXin001}, it suffices to determine $N_r=\min \{a_0 \pmod{a_1}=r : a_0\in \mathcal{R}\}$ for $0\leq r\leq a_1-1$. Now consider the case $A=(a,ha+dB)=(a,ha+db_1, ha+db_2, ..., ha+db_k)$ where $B=(b_1,b_2, ..., b_k)$. We reduce the computation of $N_{r}$ to solving a much simpler optimization problem $O_B(M)$ with parameter $M$. For some special sequences,  these optimization problems $O_B(M)$ are easy to solve,  and we may further analysis to obtain desired formulas of $g(A)$.
This approach succeeds for all the above displayed cases. It is also used to prove a conjecture about $g(A)$ for the square sequence $A=(a, a+1^2, a+2^2, ..., a+k^2)$ in \cite{Fliuxin22}. That paper also studies power sequence and prime sequence, but the resolution of $g(A)$ needs the support of Number Theory.

In some special cases, our approach can give explicit formulas of $N_r$. Such formulas can also be used to study $g(A)$ and some other statistics. For instance,
the \emph{Frobenius number} $g(A)=\max \mathcal{NR}$, the \emph{Sylvester number} $n(A)=\sum_{n\in \mathcal{NR}} 1$, the \emph{Sylvester sum}  $s(A)=\sum_{n\in \mathcal{NR}} n$, and the \emph{Sylvester power sum} $s_\mu(A)=\sum_{n\in{\mathcal{NR}}(A)}n^\mu$, which was recently introduced by Takao Komatsu. See \cite{T.Komatsu2022Arx, T.Komatsu2021, T.Komatsu2022}
and its references.

It should be pointed out that $g(A), n(A), s(A)$ and $s_\mu(A)$ are all related to $N_r$, especially to $\sum_{r=1}^{a-1}N_r^p, p\geq 1$. This article mainly focuses on $g(A), n(A)$ and $s(A)$. Basically if we obtain exact formulas for $N_r$, then we can compute their power sums. But the formulas become complicated for symbolic $a$. We find the constant term method applies if $N_r$ is a piecewise linear function about $r$. In such a situation, MacMahon's partition analysis can be used to write the generating function $f(x)=\sum_{r=0}^{a-1} x^{N_r}$ as a short sum of simple Elliott rational functions. Then the formula $(x \frac{d}{dx})^p f(x)|_{x=1}$ combined with constant term extractions can be used to give a systematical approach to the computation of $\sum_{r=1}^{a-1}N_r^p, p\geq 1$.

This paper is organized as follows. Section 1 is this introduction.
In Section 2,  we will use A. Brauer and J. E. Shockley \cite{J. E. Shockley}, E. S. Selmer \cite{E. S. Selmer} and A. Tripathi's conclusion \cite{A. Tripathi0}  to get crude formulas of $g(A), n(A), s(A)$ about $N_r$. When $A=(a, a+B)$,  we transform $N_r$ into a much simpler optimization problem $O_B(M)$.
In Section 3,  we solve the optimization problem $O_B(M)$ to obtain the $g(A), n(A), s(A)$ for some special short sequences including  $A=(a^2, a^2+1, a^2+a, a^2+a+1)$,  which is a hard instance in \cite{D. Einstein}. Meanwhile,  we will give concise proofs of many known formulas. This illustrates the wide application of our method.
In Section 4,  we use the same method to calculate some formulas $g(A), n(A), s(A)$ for long sequences,  such as $A=(a, ha+d, ha+3d, ha+5d, ..., ha+(2k+1)d)$, etc.
In Section 5, we describe how to use the constant term method to calculate $n(A)$, $s(A)$ and $s_\mu(A)$.

\section{A Reduction to a Minimization Problem}

This paper focus on the computation of $g(A), n(A)$ and $s(A)$ for some special $A$.
A. Brauer and J. E. Shockley \cite{J. E. Shockley}, E. S. Selmer \cite{E. S. Selmer}, A. Tripathi \cite{A. Tripathi0} and T. Komatsu \cite{T.Komatsu2022Arx} gave the following results respectively,  which are widely used. It is convenient to use the short hand notation
$A:=(a, B)=(a, b_1, b_2, ..., b_k)$, $\gcd (A)=1$.

\begin{lem}[\cite{J. E. Shockley}, \cite{E. S. Selmer}, \cite{A. Tripathi0}, \cite{T.Komatsu2022Arx}]\label{LiuXin001}
Let $a_0\in \mathbb{N}$, $A:=(a, B)=(a, b_1, b_2, ..., b_k)$, $\gcd (A)=1$ and
$$N_r:=N_r(a, B)=\min\{ a_0\mid a_0\equiv r\mod a, \ d(a_0;B)>0\}.$$
Then the Frobenius number, Sylvester number, Sylvester sum and Sylvester power sum are respectively:
\begin{align*}
g(A)&=g(a, B)=\max_{r\in \lbrace 0, 1, ..., a-1\rbrace}N_r -a,\\
n(A)&=n(a, B)=\frac{1}{a}\sum_{r=1}^{a-1}N_r-\frac{a-1}{2},\\
s(A)&=s(a, B)=\frac{1}{2a}\sum_{r=1}^{a-1}N_r^2
-\frac{1}{2}\sum_{r=1}^{a-1}N_r+\frac{a^2-1}{12},\\
s_{\mu}(A)&=s_{\mu}(a,B)=\frac{1}{\mu+1}\sum_{\kappa=0}^{\mu}\binom{\mu+1}{\kappa}\mathcal{B}_{\kappa}a^{\kappa-1}
\sum_{r=1}^{a-1}N_r^{\mu+1-\kappa}
+\frac{\mathcal{B}_{\mu+1}}{\mu+1}(a^{\mu+1}-1),
\end{align*}
where $\mu$ is a positive integer, $\mathcal{B}_{\kappa}$ is the Bernoulli number.
\end{lem}

Note that $N_0=0$ for all $A$,  so we need only consider $N_r$ for $1\leq r\leq a-1$. Also note that $N_r=N_{r+ka}$,  so in the formula,  we can let $r$ ranges over any set whose remainders are $\{0, 1, \dots,  a-1\}$. More precisely,  we will use the following result.
\begin{prop}\label{0201}
Let $A=(a, b_1, ..., b_k)$, $\gcd(A)=1$, $\gcd(a, d)=1, \ d\in \mathbb{P}$. Then we have
\begin{equation}
\{N_0, N_1, N_2,..., N_{a-1}\}=\{N_{d\cdot 0}, N_{d\cdot 1}, N_{d\cdot 2},..., N_{d\cdot (a-1)}\}.\label{0402}
\end{equation}
\end{prop}
\begin{proof}
When $r$ takes one of the remaining classes of $a$,  so does $dr$.
\end{proof}

We study the case $A=(a, ha+dB)=(a, ha+db_1, ..., ha+db_k)$. We find
$g(a, ha+dB)$ is closely related to a much simpler minimization problem defined by:
$$O_B(M):=\min\Big\{\sum_{i=1}^kx_i \mid \sum_{i=1}^k b_ix_i=M, \ x_i\in\mathbb{N}, 1\leq i\leq k\Big\}.$$

For the sake of convenience,  in what follows we shall always assume $x_i\in\mathbb{N}, 1\leq i\leq k$ unless specified otherwise.
\begin{lem}\label{0202}
Suppose $A=(a, ha+db_1, ..., ha+db_k)$,  $k, h, d\in\mathbb{P}$ and $\gcd(A)=1$, $m\in\mathbb{N}$,  $\gcd(a, d)=1$. For a given $ 0\leq r\leq a-1$,  we have
\begin{equation}\label{0203}
N_{dr}=\min \{O_B(ma+r) \cdot ha+(ma+r)d \mid m\in \mathbb{N}\}.
\end{equation}
\end{lem}
\begin{proof}
We have the following equations
\begin{align*}
N_{dr}&=\min\lbrace a_0\mid a_0\equiv dr\mod a;\ d(a_0;ha+db_1, ..., ha+db_k)>0\rbrace
\\&=\min\Big\{\sum_{i=1}^k(ha+db_i)x_i\mid \sum_{i=1}^k(ha+db_i)x_i\equiv dr\mod a, \ x_i\in\mathbb{N}, 1\leq i\leq k \Big\}
\\&=\min\Big\{\big(\sum_{i=1}^kx_i\big)\cdot ha+d\cdot\sum_{i=1}^kb_ix_i\mid d\sum_{i=1}^kb_ix_i\equiv dr\mod a, \ x_i\in\mathbb{N}, 1\leq i\leq k \Big\}
\\&=\min\Big\{\big(\sum_{i=1}^kx_i\big)\cdot ha+d\cdot\sum_{i=1}^kb_ix_i\mid \sum_{i=1}^kb_ix_i\equiv r\mod \frac{a}{(a, d)}, \ x_i\in\mathbb{N}, 1\leq i\leq k \Big\}
\\&=\min\Big\{\big(\sum_{i=1}^kx_i\big)\cdot ha+d(ma+r)\mid \sum_{i=1}^kb_ix_i=ma+r, \ m, x_i\in\mathbb{N}, 1\leq i\leq k \Big\}.
\end{align*}
Now for fixed $m$,  and hence fixed $M=ma+r$,
$\sum_{i=1}^kx_i$ is minimized to $O_B(ma+r)$. This completes the proof.
\end{proof}

Note that the argument in the proof also works for $h=0$. This leads to a simple proof of the following result.
\begin{thm}[\cite{J. E. Shockley}]\label{adb12kg}
Let $A=(a, dB)=(a, db_1, db_2, ..., db_k)$,  with $d\in \mathbb{P}$ and $\gcd(A)=1$. Then:
\begin{equation}
g(a, db_1, ..., db_k)=dg(a, b_1, ..., b_k)+(d-1)a. \label{0401}
\end{equation}
\end{thm}
\begin{proof}
By the definition of $N_{dr}$,  for a given $r$,  we have
\begin{align*}
N_{dr}(a, dB)&=\min \Big\{ \sum_{i=1}^kx_i(db_i)\mid \sum_{i=1}^kx_i(db_i)\equiv dr \mod a,  x_i\in\mathbb{N}, 1\leq i\leq k\Big\}
\\&=d\cdot \min \Big\{\sum_{i=1}^kx_ib_i\mid \sum_{i=1}^kx_ib_i\equiv r \mod a, x_i\in\mathbb{N}, 1\leq i\leq k\Big\}
\\&=d\cdot N_{r}(a, B).
\end{align*}
It follows that
\begin{align*}
g(a, dB)+a =\mathop{\max}\limits_{r\in \lbrace 0, 1, ..., a-1\rbrace} \lbrace N_{dr}(a, dB)\rbrace = d\mathop{\max}\limits_{r\in \lbrace 0, 1, ..., a-1\rbrace} \lbrace N_{r}(a, B)\rbrace =d (g(a, B)+a).
\end{align*}
The theorem then follows.
\end{proof}
{\it Note}: This proof is quite simple,  compared with existing proofs \cite{Ramrez Alfonsn}. Moreover,  we do not require $d=\gcd(db_1, db_2, ..., db_k)$.

For the convenience of the following discussion and by Lemma \ref{0202},  we can define an intermediate function with respect to $m$,  namely:
\begin{equation}\label{e-hd-11}
N_{dr}(m) := ha\cdot O_B(ma+r) +(ma+r)d,  \quad \text{where} \ A=(a, ha+dB).
\end{equation}

Lemma \ref{0202} suggests the following strategy for $N_{dr}$ where
$A=(a, ha+dB)$: we first try to solve $O_B(M)$ for general $M$. If we have a formula that is nice enough,  then we can analyze $N_{dr}(m)$. In fact, if $N_{dr}(m)$ increases with $m$, then we have a formula for $N_{dr}=N_{dr}(0)$. Hence we can further obtain the formula of $g(A)$, $n(A)$, and $s(A)$.

This strategy succeeds in many situations, as we shall illustrate in the next two sections.
We need one more easy fact.
\begin{prop}
If $\frac{a}{b}>0$,  then $f(x)=\frac{ax}{b}$,  $g(x)=\lfloor\frac{ax}{b}\rfloor$ and $h(x)=\lceil\frac{ax}{b}\rceil$ are increasing with respect to $x$.
\end{prop}

\section{Frobenius Formula for Some Short Sequences}
This section is divided into four subsections. In each subsection, we use our scheme to find the formulas of $g(A), n(A), s(A)$ where $A=(a,ha+dB)$ for
one typical example. The first two are for $A=(a, ha+d, ha+jd)$ and $A=(a, ha+d, ha+2d, ha+jd)$.
The corresponding $O_B(M)$ can be solved by the greedy algorithm, and the situation becomes easy when $N_{dr}=N_{dr}(0)$.
The third is for a generalization of $A=(a^2, a^2+1, a^2+a, a^2+(a+1))$. This $A$ was considered as a hard instance
from another view. The fourth is for $A=(a,ha-d,ha+d) $, where we need to solve $O_B(M)$ for integer $M$ (allowed to be negative).

\subsection{The Case of $A=(a,ha+d,ha+jd)$}
There is a large class of $B$ so that $O_B(M)$
can be achieved by the greedy algorithm.
Even in these cases $g(A)$ might be too complicated to solve.
Our scheme works for many instances.

In this subsection, we consider the case $A=(a, ha+d, ha+jd)$ which is already complicated,  though the following minimization problem is easy.
\begin{lem}
Suppose $B=(1, j)$,  $j>2$. If $M= sj+r_1$, $0\le r_1\le j-1$,  then we have
$$O_B(M)=r_1+s=M-(j-1)\lfloor M/j \rfloor,  $$
with the minimum achieved at $(x_1, x_2)=(r_1, s)$.
\end{lem}

We start by proving the following result.
\begin{thm}[\cite{Roberts1}]\label{111}
Let $A=(a, a+1, a+j)$,  $a, j>2$. Then
$$\begin{aligned}
  g(A)=
\left\{
    \begin{array}{lc}
    \frac{a(a+1)}{j}+(j-3)a-1 & \text{if}\ \ a\equiv -1 \mod j;\ \ a\geq j^2-5j+3,\ \\
    \lfloor \frac{a+1}{j}\rfloor (a+j)+(j-3)a-1 & \text{if}\ \ a\not\equiv -1 \mod j;\ \ a\geq j^2-4j+2. \\
    \end{array}
\right.
\end{aligned}$$
\end{thm}

\begin{rem}
 We first see this theorem (with a typo) in the book \cite{Ramrez Alfonsn},  which was cited as a result in \cite{Roberts1}. The theorem was not proved,  as the author said that the proof process is rather long.
\end{rem}

Now we can give a simple proof of a slight generalization of Theorem \ref{111} as follows.
\begin{thm}\label{222}
Let $A=(a, ha+d, ha+jd)$,  $a, j>2$,  and $a=kj-t,  k\geq 1,  0\leq t\leq j-1, gcd(a,d)=1, h\geq d$, where $a,h,d,j,k,t\in \mathbb{N}$. Then
$$\begin{aligned}
  g(A)=
\left\{
    \begin{array}{lc}
         \frac{ha^2}{j}+(j-2)ha+(a-1)d-a & \text{if}\ \ t=0;\ \ \ \ \ \ \ \ \ \ \ \ \ \ \ \ \ \ \ \ \ \ \ \ \ \ \ \ \ \ \ \ \ \ \\
        \frac{ha(a+1)}{j}+(j-3)ha+(a-1)d-a  & \text{if}\ \ t=1;\ \ \ \ \ \ \ \ \ \ \ \ \ \ \ \ \ \ \ \ \ \ \ \ \ \ \ \ \ \ \ \ \ \ \\
        \lfloor \frac{a}{j}\rfloor (ha+jd)+(j-2)ha-d-a & \text{if}\ \ 2\leq t\leq j-1;\ \ hk+d-ht\geq 0.
    \end{array}
\right.
\end{aligned}$$
\begin{small}
\begin{align*}
n(A)&=\frac{(a-1)(ha+d-1)}{2}-\frac{h(j-1)(a-t)}{2}(\frac{a+t}{j}-1),\ \ \text{if} \ \ hk+d-ht\geq 0,\\
s(A)&=\frac{(ha+d)^2(2a^2-3a+1)}{12}+\frac{h^2a(j-1)^2(k-1)}{6}(jk^2-\frac{jk}{2}-3tk+3t)
\\&-\frac{h(ha+d)(j-1)(k-1)}{6}(2j^2(k^2+\frac{k}{4})-2kj(3t+\frac{3}{4})+3t(t+1))
\\&-\frac{a(ha+d)(a-1)}{4}+\frac{ha(j-1)(a-t)}{4}(\frac{a+t}{j}-1)+\frac{a^2-1}{12},\ \ \text{if}\ \ hk+d-ht\geq 0.
\end{align*}
\end{small}
\end{thm}
\begin{proof}
Let $B=(1, j)$.
By Lemma \ref{0202},  for a given $r$,
if $M=ma+r=s\cdot j+r_1, \ s\geq 0,  0\leq r_1< j$,  then $s=\lfloor \frac{ma+r}{j}\rfloor$,  and
\begin{align}\label{equnder0}
N_{dr}(m)&=(s+r_1)ha+(ma+r)d=(s(1-j))ha+(ma+r)(ha+d)\notag
\\&=(ma+r)(ha+d)-\big\lfloor \frac{ma+r}{j}\big\rfloor(j-1)ha.
\end{align}

Now assume $a=kj-t,  k\geq 1,  0\leq t\leq j-1, \ a, j>2$. Then we have
\begin{align*}
N_{dr}(m+1)-N_{dr}(m)&=((m+1)a+r)(ha+d)-\big\lfloor \frac{(m+1)a+r}{j}\big\rfloor(j-1)ha
\\&\ \ \ \ -(ma+r)(ha+d)+\big\lfloor \frac{ma+r}{j}\big\rfloor(j-1)ha
\\&=a(ha+d)-(j-1)ha\Big(\big\lfloor \frac{(m+1)a+r}{j}\big\rfloor-\big\lfloor \frac{ma+r}{j}\big\rfloor\Big)
\\&\geq a(ha+d)-(j-1)ha\cdot\lceil \frac{a}{j}\rceil
\\&= a(hkj-ht+d)-(j-1)hak =a(hk+d-ht).
\end{align*}
It follows that if $hk+d-ht\ge 0$,  then $N_{dr}(m)$ is increasing,
and hence minimizes at $m=0$.
Now, we write
$$N_{dr}=N_{dr}(0)=(s+r_1)ha+rd=(s+r_1)ha+(sj+r_1)d = s(ha+jd)+(ha+d) r_1, $$
which is increasing with respect to $s$ and to $r_1$,  respectively.

By $1\leq r\leq a-1=kj-t-1$,  we see that $s\le k-1$. We need to consider the following three cases of $t$.

i) When $t=0$,  the condition $hk+d-ht\ge 0$ always holds true,  and $(s, r_1)=(k-1, j-1)$ gives rise $\max \{N_{dr}\} = (k-1)(ha+jd)+(ha+d)(j-1)=ha^2/j+(j-2)ha+(a-1)d.$

ii) When $t=1$,  the condition $hk+d-ht\ge 0$ always holds true. We need to consider two values: $(s, r_1)=(k-1, j-2)$
gives rise $N_{d(a-1)}=(k-1)(ha+jd)+(ha+d)(j-2)$; $(s, r_1)=(k-2, j-1)$,  which gives rise $N_{d(a-j)}=(k-2)(ha+jd)+(ha+d)(j-1)$. Clearly,  the former gives the maximum: $\max \{N_{dr}\} =(a+1)ha/j+(j-3)ha+(a-1)d$.

iii) When $t\ge 2$,  we need the condition $hk+d-ht\ge 0$ to ensure that $N_{dr}(m)$ is increasing. We need to consider two values:
$(s, r_1)=(k-1, j-t-1)$ gives rise $N_{d(a-1)}=(ha+jd)(k-1)+(ha+d)(j-t-1)$;
$(s, r_1)=(k-2, j-1)$ gives rise $N_{d(a-j+t-1)}=(ha+jd)(k-2)+(ha+d)(j-1)$.
Since $N_{d(a-1)}-N_{d(a-j+t-1)}=(ha+jd)-t(ha+d)=(j-t)d-(t-1)ha\leq 0$ by $d\leq h$, we have $\max \{N_{dr}\} = (ha+jd)(k-2)+(ha+d)(j-1)=\lfloor \frac{a}{j}\rfloor(ha+jd)+(j-2)ha-d$.

From $g(A)=\max \{ N_{dr}\}-a$, we get the formula of $g(A)$ in this theorem.

Next we use the alternative formula by \eqref{equnder0}:
$$N_{dr}=N_{dr}(0)=r(ha+d)-\lfloor \frac{r}{j}\rfloor(j-1)ha.$$
For $n(A)$, we have
\begin{small}
\begin{align*}
\sum_{r=1}^{a-1}N_{dr}&=\sum_{r=1}^{a-1}\big(r(ha+d)-\lfloor \frac{r}{j}\rfloor(j-1)ha\big)
\\&=\frac{a(ha+d)(a-1)}{2}-(j-1)ha\sum_{r=0}^{a-1}\lfloor\frac{r}{j}\rfloor
\\&=\frac{a(ha+d)(a-1)}{2}-(j-1)ha((1+2+\cdots +k-2)j+(j-t)(k-1))
\\&=\frac{a(ha+d)(a-1)}{2}-\frac{ha(j-1)(a-t)}{2}\big(\frac{a+t}{j}-1\big),
\end{align*}
\end{small}
which implies that
\begin{small}
\begin{align*}
n(A)&=\frac{1}{a}\sum_{r=1}^{a-1}N_{dr}-\frac{a-1}{2}
=\frac{(a-1)(ha+d-1)}{2}-\frac{h(j-1)(a-t)}{2}\big(\frac{a+t}{j}-1\big).
\end{align*}
\end{small}

For $s(A)$, we have
\begin{small}
\begin{align*}
\sum_{r=1}^{a-1}N_{dr}^2&=\sum_{r=1}^{a-1}\Big(r(ha+d)-\lfloor\frac{r}{j}\rfloor(j-1)ha\Big)^2
\\&=\sum_{r=1}^{a-1}r^2(ha+d)^2+\sum_{r=1}^{a-1}(\lfloor\frac{r}{j}\rfloor)^2(j-1)^2(ha)^2
-2ha(ha+d)(j-1)\sum_{r=0}^{a-1}r\lfloor \frac{r}{j}\rfloor
\\&=\frac{a(ha+d)^2(2a^2-3a+1)}{6}+\frac{(ha)^2(j-1)^2(k-1)}{3}\Big(jk^2-\frac{jk}{2}-3tk+3t\Big)
\\&-\frac{ha(ha+d)(j-1)(k-1)}{3}\Big(2j^2(k^2+\frac{k}{4})-2kj(3t+\frac{3}{4})+3t(t+1)\Big),
\end{align*}
\end{small}
which implies that
\begin{small}
\begin{align*}
s(A)=&\frac{1}{2a}\sum_{r=1}^{a-1}N_{dr}^2-\frac{1}{2}\sum_{r=1}^{a-1}N_{dr}+\frac{a^2-1}{12}
\\=&\frac{(ha+d)^2(2a^2-3a+1)}{12}+\frac{h^2a(j-1)^2(k-1)}{6}\Big(jk^2-\frac{jk}{2}-3tk+3t\Big)
\\&-\frac{h(ha+d)(j-1)(k-1)}{6}\Big(2j^2(k^2+\frac{k}{4})-2kj(3t+\frac{3}{4})+3t(t+1)\Big)
\\&-\frac{a(ha+d)(a-1)}{4}+\frac{ha(j-1)(a-t)}{4}\big(\frac{a+t}{j}-1\big)+\frac{a^2-1}{12}.
\end{align*}
\end{small}
We have thus proved the theorem.
\end{proof}
Further more, we can calculate
\begin{small}
\begin{align*}
\sum_{r=1}^{a-1}N_{dr}^p&=\sum_{r=1}^{a-1}\big(r(ha+d)-\lfloor\frac{r}{j}\rfloor(j-1)ha\big)^p,
\end{align*}
\end{small}
thus obtain a formula for $s_{\mu}(A)$. But such a formula is complicated.

It is not hard to show that Theorem \ref{111} is a corollary of Theorem \ref{222}.
Here we briefly explain some of the details.  Let $h=d=1$. When $2\leq t\leq j-1$,  we have $a\not\equiv -1 \mod j$. Then we can get $a\geq j^2-4j+2$ from $ k+1-t\geq 0$.
It is not difficult to find in the proof of Theorem \ref{222} that we need $a(a+1)-(j-1)a(\lceil \frac{a}{j}\rceil)\geq 0$.
If $t\leq j-2$,  we have
$a(a+1)-(j-1)a(\lceil \frac{a}{j}\rceil)=a(a+1)-(j-1)a(\frac{a+t}{j})$,  it is decreasing with respect to $t$. When $t=j-2$,  $a(a+1)-(j-1)a\cdot \frac{a+j-2}{j}\geq 0$ implies
$a\geq j^2-4j+2$.
If $t=j-1$,  similarly,  we have $a\geq j^2-3j+1$. Due to the condition $a=kj-(j-1)$,  we can reduce this lower bound by $(j-1)$. So we have $a\geq j^2-3j+1-(j-1)$. Therefore  Theorem \ref{222} is more precise than Theorem \ref{111}.

The special case $a \equiv 0 \mod j$ and  $h=d=1$ of Theorem \ref{222} can be restated as follows.
\begin{cor}\label{333}
Let $A=(sa, sa+1, sa+a)$,  $a>2$,  $s\ge 1$. Then
$g(A)=as(a+s-2)-1.$
\end{cor}
The special case when $s=a$ of the above result appears in \cite{M. Hujter},  i.e.,  $g(a^2, a^2+1, a^2+a)=2a^3-2a^2-1$.
This formula can be easily obtained by Theorem \ref{adb12kg}:
\begin{align*}
  g(sa, sa+1, sa+a) &= a g(s, s+1, sa+1) +(a-1)(as+1)=a g(s, s+1)+(a-1)(as+1).
\end{align*}

\subsection{The Case of $A=(a,ha+d,ha+2d,ha+jd)$}\label{hard12jj}

Using our method,  we obtain the following interesting formula.

\begin{thm}\label{012j}
Let $A=(a, ha+d, ha+2d, ha+jd)$,  $a, j, h, d\in \mathbb{P}, h\geq d$,  $a\geq 2$ and $j\geq 4$. Moreover let $a=kj-t$,  where $k\geq 1$ and $0\leq t\leq j-1$,  and we require $k+1-\lceil \frac{t}{2}\rceil\geq 0$.
Then
\begin{align*}
g(A)&=\max\Big\{ha\big(\lfloor \frac{a-1}{j}\rfloor +\lceil \frac{a-1}{2}-\frac{j}{2}\lfloor \frac{a-1}{j}\rfloor\rceil\big)+(a-1)d-a,
\\&\ \ \ \ ha\big(\lfloor\frac{a-1}{j}\rfloor+\lceil\frac{j-1}{2}\rceil-1\big)+\big(j\lfloor \frac{a-1}{j}\rfloor-1\big)d-a\Big\}.
\end{align*}
Furthermore
$$\begin{small}\begin{aligned}
g(A)=
\left\{
\begin{aligned}
ha\Big(\Big\lfloor \frac{a-1}{j}\Big\rfloor +\bigg\lceil \frac{a-1}{2}-\frac{j}{2}\Big\lfloor \frac{a-1}{j}\Big\rfloor\bigg\rceil\Big)+(a-1)d-a\ \ & \text{if}\ \bigg\lceil\frac{a-1}{2}-\frac{j}{2}\Big\lfloor \frac{a-1}{j}\Big\rfloor\bigg\rceil\geq \Big\lceil\frac{j-1}{2}\Big\rceil-1, \\
    ha\Big(\Big\lfloor\frac{a-1}{j}\Big\rfloor+\Big\lceil\frac{j-1}{2}\Big\rceil-1\Big)+\Big(j\Big\lfloor \frac{a-1}{j}\Big\rfloor-1\Big)d-a\ \ & \text{otherwise}.\\
\end{aligned}
\right.
\end{aligned}\end{small}$$
We have
$$\begin{tiny}\begin{aligned}
n(A)=
\left\{
\begin{aligned}
&\frac{(a-1)(d-1)}{2}+\frac{h(k-1)}{2}(jk-2t)+\frac{h(j-1)(j+1)(k-1)}{4}+\frac{h(j-t-1)(j-t+1)}{4}\ & \text{if}\ \ j-1\  \text{is even}, t\  \text{is even}, \\
&\frac{(a-1)(d-1)}{2}+\frac{h(k-1)}{2}(jk-2t)+\frac{h(j-1)(j+1)(k-1)}{4}+\frac{h(j-t)^2}{4} \ & \text{if}\ \ j-1\ \text{is even}, t\  \text{is odd},\ \\
&\frac{(a-1)(d-1)}{2}+\frac{h(k-1)}{2}(jk-2t)+\frac{hj^2(k-1)}{4}+\frac{h(j-t)^2}{4}\ & \text{if}\ \ j-1\ \text{is odd}, t\  \text{is even},\ \\
&\frac{(a-1)(d-1)}{2}+\frac{h(k-1)}{2}(jk-2t)+\frac{hj^2(k-1)}{4}+\frac{h(j-t-1)(j-t+1)}{4}\ & \text{if}\ \ j-1\ \text{is odd}, t\  \text{is odd}.\ \
\end{aligned}
\right.
\end{aligned}\end{tiny}$$
\end{thm}
\begin{rem}
Recently, the formula $g(A)$ for $A=(a,a+d,..., a+kd, a+Kd)$ (i.e., a arithmetic sequence with an additional term) was studied in \cite{Rodseth}. The cases $j=4, 5, 6$ and $h=d=1$ (stated in Corollary \ref{012456})  were studied by A. L. Dulmage and N. S. Mendelsohn \cite{A. L. Dulmage} using graphical methods for $g(A)$.
\end{rem}

Now we provide a short proof for the general case.
\begin{proof}
By Lemma \ref{0202},  we need to solve $O_B(M)=\min\{x_1+x_2+x_3 \mid x_1+2x_2+jx_3=M\}$.
For a given $r$,  assume $M=ma+r=js+r_1$,  where$\ s\geq 0,  0\leq r_1< j$. Then
$$O_B(M)=O_B(ma+r)=s+\big\lceil \frac{r_1}{2} \big\rceil = \big\lfloor \frac{ma+r}{j} \big\rfloor + \big\lceil \frac{r_1}{2} \big\rceil, $$
and
$$N_{dr}(m)=\Big(s+\big\lceil \frac{r_1}{2} \big\rceil\Big)ha+(ma+r)d=\Big(\big\lfloor \frac{ma+r}{j} \big\rfloor +\big\lceil \frac{r_1}{2} \big\rceil\Big)ha+(ma+r)d, $$
where $r_1$ is related to $m$. We next prove that $N_{dr}(m)$ increases with respect to $m$.

By $ma+r=js+r_1$ and $a=kj-t$ with $k\geq 1$,  $0\leq t\leq j-1$. We have
$(m+1)a+r=js^{\prime}+r_1^{\prime}=j(s+k)+r_1-t$,  where $0\leq r_1^{\prime}\leq j-1$.
If $r_1<t$,  we have $r_1^{\prime}=r_1+j-t>r_1$ and $s^{\prime}=s+k-1\geq s$. Obviously,  we have $N_{dr}(m+1)\geq N_{dr}(m)$. If $r_1\geq t$,  we have $r_1^{\prime}=r_1-t\leq r_1$ and $s^{\prime}=s+k>s$. Now consider
$$\begin{aligned}
N_{dr}(m+1)-N_{dr}(m)&=\Big(k+\big\lceil \frac{r_1-t}{2}\big\rceil-\big\lceil \frac{r_1}{2}\big\rceil\Big)ha+ad
\\&\geq \Big(k+1-\big\lceil \frac{t}{2}\big\rceil\Big)ad \geq 0.
\end{aligned}$$
Therefore $N_{dr}(m)$ increases with respect to $m$.

It follows that if $r=sj+r_1$ with $0\leq r_1< j$,  then
$$ N_{dr}=N_{dr}(0)=ha \cdot O_B(r)+rd =ha\Big(s + \big\lceil \frac{r_1}{2} \big\rceil\Big) + (sj+r_1)d=(ha+jd)s+ha\big\lceil \frac{r_1}{2} \big\rceil+dr_1.$$
This is increasing with respect to $s$ and to $r_1$,  respectively.
If $a-1= s_0 j + r_0$ with $0\le r_0 \le j-1$,  then we need only consider two cases:

$\bullet$ $(s, r_1)=(s_0, r_0)$,  which gives rise
$N_{d(a-1)}=ha\big(s_0 + \big\lceil \frac{r_0}{2} \big\rceil\big) + (a-1)d $;

$\bullet$ $(s, r_1)=(s_0-1, j-1)$,  which gives rise
$N_{d(a-r_0-2)}=ha\big(s_0-1+\big\lceil \frac{j-1}{2} \big\rceil\big)+(a-r_0-2)d$.

Since $s_0=\big\lfloor \frac{a-1}{j}\big\rfloor$ and $r_0=a-1-j\big\lfloor \frac{a-1}{j}\big\rfloor$, we have
\begin{align*}
g(A)&=\max\{N_{dr}\}-a =\max\{N_{d(a-1)}-a, \ \ N_{d(a-r_0-2)}-a\}
\\&=\max\Big\{ha\big(s_0+\big\lceil\frac{r_0}{2}\big\rceil\big)+(a-1)d-a, \ \ ha\big(s_0+\big\lceil\frac{j-1}{2}\big\rceil-1\big)+(a-r_0-2)d-a\Big\}
\\&=\max\Big\{ha\Big(\big\lfloor \frac{a-1}{j}\big\rfloor +\Big\lceil \frac{a-1}{2}-\frac{j}{2}\big\lfloor \frac{a-1}{j}\big\rfloor\Big\rceil\Big)+(a-1)d-a,
\\&\ \ \ \ \ \ \ \ \ \ ha\Big(\big\lfloor\frac{a-1}{j}\big\rfloor+\big\lceil\frac{j-1}{2}\big\rceil-1\Big)+\big(j\big\lfloor \frac{a-1}{j}\big\rfloor-1\big)d-a\Big\}.
\end{align*}
Now it is straightforward to check that:
If $\big\lceil\frac{r_0}{2}\big\rceil \geq \big\lceil\frac{j-1}{2}\big\rceil-1$, i.e.,  $\Big\lceil\frac{a-1}{2}-\frac{j}{2}\big\lfloor \frac{a-1}{j}\big\rfloor\Big\rceil \geq \big\lceil\frac{j-1}{2}\big\rceil-1$,  then we have
$$g(A)=N_{d(a-1)}-a=ha\Big(\Big\lfloor \frac{a-1}{j}\Big\rfloor +\bigg\lceil \frac{a-1}{2}-\frac{j}{2}\Big\lfloor \frac{a-1}{j}\Big\rfloor\bigg\rceil\Big)+(a-1)d-a;$$
Otherwise,
$$g(A)=N_{d(a-r_0-2)}-a=ha\Big(\Big\lfloor\frac{a-1}{j}\Big\rfloor+\Big\lceil\frac{j-1}{2}\Big\rceil-1\Big)+\Big(j\Big\lfloor \frac{a-1}{j}\Big\rfloor-1\Big)d-a.$$

Next let $r=js+r_1$,  where$\ s\geq 0,  0\leq r_1< j$ and write
$$N_{dr}=N_{dr}(0)=\big(\big\lfloor \frac{r}{j} \big\rfloor +\big\lceil \frac{r_1}{2} \big\rceil\big)ha+rd.$$
By $a-1=(k-1)j+(j-t-1)$,  where $k\geq 1$ and $0\leq t\leq j-1$. The computation for $n(A)$
is divided into four cases according to the parity of $j$ and $t$.

Case 1: If $j-1$ is even, $t$ is even, we have
\begin{small}
\begin{align*}
\sum_{r=1}^{a-1}N_{dr}&=\sum_{r=1}^{a-1}\big(\lfloor\frac{r}{j}\rfloor +\lceil\frac{r_1}{2}\rceil\big)ha+rd
\\&=\sum_{r=0}^{a-1}dr+ha\bigg(\sum_{i=1}^{k-2}j+(k-1)(j-t)+\big(1+\cdots +\frac{j-1}{2}\big)2(k-1)+\big(1+\cdots+\frac{j-t-1}{2}\big)2\bigg)
\\&=\frac{a(a-1)d}{2}+\frac{ha(k-1)}{2}(jk-2t)+\frac{ha(j-1)(j+1)(k-1)}{4}
+\frac{ha(j-t-1)(j-t+1)}{2}.
\end{align*}
\end{small}
Case 2: If $j-1$ is even, $t$ is odd, we have
\begin{small}
\begin{align*}
\sum_{r=1}^{a-1}N_{dr}&=\sum_{r=1}^{a-1}\big(\lfloor\frac{r}{j}\rfloor +\lceil\frac{r_1}{2}\rceil\big)ha+rd
\\&=\frac{a(a-1)d}{2}+\frac{ha(k-1)}{2}(jk-2t)+\frac{ha(j-1)(j+1)(k-1)}{4}+\frac{ha(j-t)^2}{4}.
\end{align*}
\end{small}
Case 3: If $j-1$ is odd, $t$ is even, we have
\begin{small}
\begin{align*}
\sum_{r=1}^{a-1}N_{dr}&=\sum_{r=1}^{a-1}\big(\lfloor\frac{r}{j}\rfloor +\lceil\frac{r_1}{2}\rceil\big)ha+rd
\\&=\frac{a(a-1)d}{2}+\frac{ha(k-1)}{2}(jk-2t)+\frac{haj^2(k-1)}{4}+\frac{ha(j-t)^2}{4}.
\end{align*}
\end{small}
Case 4: If $j-1$ is odd, $t$ is odd, we have
\begin{small}
\begin{align*}
\sum_{r=1}^{a-1}N_{dr}&=\sum_{r=1}^{a-1}\big(\lfloor\frac{r}{j}\rfloor +\lceil\frac{r_1}{2}\rceil\big)ha+rd
\\&=\frac{a(a-1)d}{2}+\frac{ha(k-1)}{2}(jk-2t)+\frac{haj^2(k-1)}{4}+\frac{ha(j-t-1)(j-t+1)}{4}.
\end{align*}
\end{small}
By Lemma \ref{LiuXin001}, we have
$$\begin{tiny}\begin{aligned}
n(A)&=\frac{1}{a}\sum_{r=1}^{a-1}N_{dr}-\frac{a-1}{2}
\\&=
\left\{
\begin{aligned}
&\frac{(a-1)(d-1)}{2}+\frac{h(k-1)}{2}(jk-2t)+\frac{h(j-1)(j+1)(k-1)}{4}+\frac{h(j-t-1)(j-t+1)}{4}\ & \text{if}\ \ j-1\  \text{is even}, t\  \text{is even}, \\
&\frac{(a-1)(d-1)}{2}+\frac{h(k-1)}{2}(jk-2t)+\frac{h(j-1)(j+1)(k-1)}{4}+\frac{h(j-t)^2}{4} \ & \text{if}\ \ j-1\ \text{is even}, t\  \text{is odd},\ \\
&\frac{(a-1)(d-1)}{2}+\frac{h(k-1)}{2}(jk-2t)+\frac{hj^2(k-1)}{4}+\frac{h(j-t)^2}{4}\ & \text{if}\ \ j-1\ \text{is odd}, t\  \text{is even},\ \\
&\frac{(a-1)(d-1)}{2}+\frac{h(k-1)}{2}(jk-2t)+\frac{hj^2(k-1)}{4}+\frac{h(j-t-1)(j-t+1)}{4}\ & \text{if}\ \ j-1\ \text{is odd}, t\  \text{is odd}.\ \
\end{aligned}
\right.
\end{aligned}\end{tiny}$$
This completes the proof.
\end{proof}
For $\sum_{r=1}^{a-1}N_{dr}^2$, we have
\begin{small}
\begin{align*}
\sum_{r=1}^{a-1}N_{dr}^2&=\sum_{r=1}^{a-1}\Big((\lfloor\frac{r}{j}\rfloor +\lceil\frac{r_1}{2}\rceil)ha+rd\Big)^2
\\&=\sum_{r=1}^{a-1}d^2r^2+h^2a^2\sum_{r=1}^{a-1}\Big(\lfloor \frac{r}{j}\rfloor+\lceil \frac{r_1}{2}\rceil\Big)^2+2had\sum_{r=1}^{a-1}r\Big(\lfloor \frac{r}{j}\rfloor+\lceil \frac{r_1}{2}\rceil\Big).
\end{align*}
\end{small}
In particular, the simplification of $\sum_{r=1}^{a-1}r(\lfloor \frac{r}{j}\rfloor+\lceil \frac{r_1}{2}\rceil)$ is difficult. Therefore, a finer simplification of $s(A)$ is difficult. However, we will provide a fast algorithm for calculating $s(A)$ in later section.

\begin{cor}{\em \cite{A. L. Dulmage}}\label{012456}
$\bullet$ Let $A=(a, a+1, a+2, a+4)$ with $a\geq 2$.  Then
$$g(A)=(a+1)\lfloor \dfrac{a}{4}\rfloor+\lfloor \dfrac{a+1}{4}\rfloor+2\lfloor\dfrac{a+2}{4}\rfloor -1.$$

$\bullet$ Let $A=(a, a+1, a+2, a+5)$ with $a\geq 2$.  Then
$$g(A)=a\lfloor \dfrac{a+1}{5}\rfloor+\lfloor \dfrac{a}{5}\rfloor+\lfloor\dfrac{a+1}{5}\rfloor+\lfloor\dfrac{a+2}{5}\rfloor+2\lfloor \dfrac{a+3}{5}\rfloor -1.$$

$\bullet$ Let $A=(a, a+1, a+2, a+6)$ with $a\geq 2$.  Then
$$g(A)=a\lfloor \dfrac{a}{6}\rfloor+2\lfloor \dfrac{a}{6}\rfloor+2\lfloor\dfrac{a+1}{6}\rfloor+5\lfloor\dfrac{a+2}{6}\rfloor+\lfloor \dfrac{a+3}{6}\rfloor+\lfloor\dfrac{a+4}{6}\rfloor+\lfloor\dfrac{a+5}{6}\rfloor-1.$$
\end{cor}
\begin{proof}
By Theorem \ref{012j} with $j=4, h=d=1$. We have $0\leq t\leq 3$,  therefore $k+1-\lceil \frac{t}{2}\rceil \geq 0$ satisfies the condition of Theorem \ref{012j}.
If $a=4k,  k\geq1$,  then $g(A)=(a^2+4a-4)/4$ is consistent with the above formula. Similarly,  $g(A)$ can be computed for $a=4k-1, \ a=4k-2, \ a=4k-3$. The results are consistent with the above. In the same way,  the reader can check that for $j=5, 6$,  the above formula holds. It is worth noting that when $j=6$,  we have $0\leq t\leq 5$. The condition $k+1-\lceil \frac{t}{2}\rceil\geq 0$ obviously holds when $0\leq t\leq 4$, but holds
when $t=5$ only for $a\geq 2$ which implies $k\geq 2$.
\end{proof}

\subsection{The Case of $A=(a^2,ha^2+d,ha^2+ad,ha^2+(a+1)d)$}
There are many cases the corresponding $O_B(M)$ cannot reach the optimal solution by the greedy algorithm. For example,  when  $B=(1, a, a+1)$,  we have
$$O_B(M)= \min \{x_1+x_2+x_3: x_1+a x_2+(a+1)x_3 =M \}.$$
Take $M=2a, a>2$. The greedy algorithm gives $(x_1, x_2, x_3)=(a-1, 0, 1)$,  but the optimal solution is obviously at $(x_1, x_2, x_3)=(0, 2, 0)$. Indeed,  we have the following result.

\begin{lem}\label{1aa1}
Suppose $B=(1, a, a+1)$,  $a\ge 2$,  $M>0$. If $(a+1)\nmid M$ and $M\geq a(\lfloor \frac{M}{a+1}\rfloor+1)$,  then we have $O_B(M)=1+\lfloor \frac{M}{a+1}\rfloor$;
Otherwise,  we have $O_B(M)=M-a\lfloor \frac{M}{a+1}\rfloor$.
\end{lem}
\begin{proof}
Suppose $M=s(a+1)+r_1$ where $0\le r_1 \le a$. We claim that if the minimum occurs at $(x_1, x_2, x_3)$,  then $x_1=0$ or $x_2=0$ or both. If not,  say $x_1$ and $x_2$ are both positive,  then
$(x_1', x_2', x_3')=(x_1-1, x_2-1, x_3+1)$ satisfies $x_1'+ax_2'+(a+1)x_3'=M$ but has smaller $x_1'+x_2'+x_3'$,  a contradiction.

Firstly,  since $(a+1)\nmid M$ is equivalent to $r_1\neq 0$,  we have $x_1=0,  x_2\neq 0$ or $x_1\neq 0,  x_2=0$. Since
$M\geq a(\lfloor \frac{M}{a+1}\rfloor+1)$ is equivalent to $s+r_1-a\geq 0$,  we have
$M=s(a+1)+r_1=(s-a+r_1)(a+1)+(a-r_1+1)a$. At this time,  we have $x_1=0, x_2\neq 0$ and
the minimum occurs at $(0,  a-r_1+1 ,  s-a+r_1)$. Therefore $O_B(M)=s+1=1+\lfloor \frac{M}{a+1}\rfloor$.

Secondly,  for other cases,  it is easy to see that $x_1\neq 0$ so that the minimum occurs at $(r_1, 0, s)$. Therefore $O_B(M)=s+r_1=M-a\lfloor \frac{M}{a+1}\rfloor$.
\end{proof}

When $M=ma^2+r, m\in\mathbb{N}$,  $O_B(M)$ has a better result. This allows us to further discuss $N_r(m)$,  even the Frobenius formula $g(A)$. The formula in the following theorem seems new. If $h=d=1$, we have $A=(a^2, a^2+1, a^2+a, a^2+a+1)$. This case was considered in \cite{D. Einstein} as a hard instance for their geometric construction.
\begin{thm}
Let $A=(a^2, ha^2+d, ha^2+ad, ha^2+(a+1)d)$,  $a>1$, $gcd(a,d)=1$. Then
\begin{align*}
g(A)&=ha^3+(d-h-1)a^2-d,\\
n(A)&=\frac{2}{3}ha^3+\frac{1}{2}(d-h-1)a^2-\frac{1}{6}ha+\frac{1}{2}(1-d),\\
s(A)&=\frac{1}{24}\Big(6h^2a^6+(9dh-8h^2-8h)a^5+(4d^2-5dh-6d+6h+2)a^4
\\&\ \ \ \ +(2h^2-11dh+2h)a^3+(5dh-6d^2+6d)a^2+2dha+2d^2-2\Big).
\end{align*}
\end{thm}
\begin{proof}
By Lemma \ref{0202},  $O_B(M)=\min\{x_1+x_2+x_3 \mid x_1+x_2a+x_3(a+1)=M\}$.
Suppose
$$M=ma^2+r=(a+1)\cdot s+r_1, $$
where$\ s\geq 0,  0\leq r_1< a+1$. By Lemma \ref{1aa1},  we have the following two types:

$\bullet$ If $r_1\neq 0$ and $a-s\leq r_1\leq a$,  we have $O_B(M)=s+1$ and
$N_{dr}(m)=(s+1)ha^2+(ma^2+r)d=(\lfloor\frac{ma^2+r}{a+1}\rfloor+1)ha^2+(ma^2+r)d$.

$\bullet$ If $0\leq r_1<a-s$,  we have $O_B(M)=s+r_1$ and
$N_{dr}(m)=(s+r_1)ha^2+(ma^2+r)d=(\lfloor\frac{ma^2+r}{a+1}\rfloor+r_1)ha^2+(ma^2+r)d$.

For the above two types,  as $m$ increases by $1$,  the value of $s$ at least increases by
$\lfloor \frac{a^2}{a+1}\rfloor=\lfloor \frac{(a+1)(a-1)+1}{a+1}\rfloor=a-1$. Obviously,  $N_{dr}(m)$ is increasing with respect to $m$. So,  we have $N_{dr}=N_{dr}(0)=O_B(r)ha^2+rd$.
When $r=a^2-1=(a+1)(a-1)$,  both $O_B(r)=a-1$ and $r$ reach the maximum.
Then we have
$\max\{ N_{dr}\}=(a-1)ha^2+(a^2-1)d=ha^3+(d-h)a^2-d$, and $g(A)=\max \{ N_{dr}\}-a^2=ha^3+(d-h-1)a^2-d$.

Next we rewrite $N_{dr}=N_{dr}(0)$ as follows.

$\bullet$ If $r_1\neq 0$ and $a-s\leq r_1\leq a$,  we have
$N_{dr}=(\lfloor\frac{r}{a+1}\rfloor+1)ha^2+rd$.

$\bullet$ If $0\leq r_1<a-s$,  we have
$N_{dr}=(\lfloor\frac{r}{a+1}\rfloor+r_1)ha^2+rd$.

Then for $n(A)$, we have
\begin{small}
\begin{align*}
\sum_{r=1}^{a^2-1}N_{dr}&=\sum_{r=1}^{a^2-1}\Big((\lfloor \frac{r}{a+1}\rfloor+1)ha^2+rd\Big)+\sum_{r=1\atop 0\leq r_1< a-s}^{a^2-1}(r_1-1)ha^2
\\&=ha^2((1+\cdots +(a-1))(a+1)+a-1)+\sum_{r=1}^{a^2-1}rd+ha^2\Big(\sum_{m=0}^{a-1}\sum_{i=0}^{a-m-1}i\Big)-ha^2\Big(\sum_{r=1\atop 0\leq r_1< a-s}^{a^2-1}1-1\Big)
\\&=ha^2\Big(\frac{(a-1)a(a+1)}{2}+a\Big)+\frac{da^2(a^2-1)}{2}+ha^2\Big(\sum_{m=0}^{a-1}\frac{(a-m)(a-m-1)}{2}\Big)
-\frac{ha^3(a+1)}{2}
\\&=\frac{2}{3}ha^5+\frac{1}{2}(d-h)a^4-\frac{1}{6}ha^3-\frac{1}{2}da^2,
\end{align*}
\end{small}
which implies that
\begin{align*}
n(A)=\frac{1}{a^2}\sum_{r=1}^{a^2-1}N_{dr}-\frac{a^2-1}{2}
=\frac{2}{3}ha^3+\frac{1}{2}(d-h-1)a^2-\frac{1}{6}ha+\frac{1}{2}(1-d).
\end{align*}

For $s(A)$, we omit some calculations and obtain
\begin{small}
\begin{align*}
\sum_{r=1}^{a^2-1}N_{dr}^2&=\sum_{m=0}^{a-1}\sum_{i=0}^{a-m-1}((m+i)ha^2+(m(a+1)+i)d)^2
+\sum_{m=0}^{a-2}\sum_{i=0}^{m}((m+1)ha^2+((m+1)a+i)d)^2
\\&=\frac{a^2}{12}(6h^2a^6+(9dh-8h^2)a^5+(4d^2-5dh)a^4+(2h^2-11dh)a^3+(5dh-6d^2)a^2+2dha+2d^2),
\end{align*}
\end{small}
which implies that
\begin{small}
\begin{align*}
s(A)&=\frac{1}{2a^2}\sum_{r=1}^{a^2-1}N_{dr}^2-\frac{1}{2}\sum_{r=1}^{a^2-1}N_{dr}+\frac{a^4-1}{12}
\\&=\frac{1}{24}\big(6h^2a^6+(9dh-8h^2-8h)a^5+(4d^2-5dh-6d+6h+2)a^4+(2h^2-11dh+2h)a^3
\\&\ \ \ \ +(5dh-6d^2+6d)a^2+2dha+2d^2-2\big).
\end{align*}
\end{small}
The proof of the theorem is now complete.
\end{proof}

Further more, we can calculate
\begin{small}
\begin{align*}
\sum_{r=1}^{a^2-1}N_{dr}^p&=\sum_{m=0}^{a-1}\sum_{i=0}^{a-m-1}\big((m+i)ha^2+(m(a+1)+i)d\big)^p
+\sum_{m=0}^{a-2}\sum_{i=0}^{m}\big((m+1)ha^2+((m+1)a+i)d\big)^p,
\end{align*}
\end{small}
thus obtain a formula for $s_{\mu}(A)$, but such a formula is complicated.

\subsection{The Case of $A=(a,ha-d, ha+d)$}
We conclude this subsection by the following theorem with possible negative $M$.
\begin{thm}\label{wehaha}
Let $A=(a, ha-d, ha+d)$,  $d, h\in \mathbb{P}, \gcd(a, d)=1$, $ha-d>1$ and $s=\lfloor \frac{ha-d}{2h}\rfloor$. Then:
\begin{align*}
g(A)&=\max\Big\lbrace \lfloor \dfrac{ha-d}{2h}\rfloor(ha+d)-a, \big(a-\lceil \dfrac{ha-d}{2h}\rceil\big)(ha-d)-a\Big\rbrace,\\
n(A)&=\frac{ha+d}{2a}s(s+1)+\frac{ha-d}{2a}(a-s)
(a-s-1)-\frac{a-1}{2},\\
s(A)&=\Big(\frac{(ha+d)^2(2s+1)}{12a}-\frac{ha+d}{4}\Big)s(s+1)
\\&\ \ \ \ +\Big(\frac{(ha-d)^2(2a-2s-1)}{12a}-\frac{ha-d}{4}\Big)(a-s)
(a-s-1)+\frac{a^2-1}{12},\\
s_{\mu}(A)&=\frac{1}{\mu+1}\sum_{\kappa=0}^{\mu}\binom{\mu+1}{\kappa}\mathcal{B}_{\kappa}a^{\kappa-1}
\Big((ha+d)^{\mu+1-\kappa}\sum_{r=0}^{s}r^{\mu+1-\kappa}
+(ha-d)^{\mu+1-\kappa}\sum_{r=1}^{a-s-1}r^{\mu+1-\kappa}\Big)
\\&\ \ \ \ +\frac{\mathcal{B}_{\mu+1}}{\mu+1}(a^{\mu+1}-1).
\end{align*}
where $\mu$ is a positive integer, $\mathcal{B}_{\kappa}$ is the Bernoulli number.
\end{thm}
\begin{proof}
By Lemma \ref{0202}, we need to solve $O_B(M)=\min\{x_1+x_2 \mid x_2-x_1=M\}$ for $M=ma+r, m\in \mathbb{Z}$. We have to distinguish two cases: 1) If $m\ge 0$,  then $x_1+x_2=2x_1+M$ minimizes to $M$ at $x_1=0$,  so that $N_{dr}$ minimizes to $T_1:=(ha+d)r$ at $m=0, \ x_1=0$. 2) If $m<0$ then $M<0$ and $x_1+x_2=2x_2-M$ minimizes to $-M$ at $x_2=0$,  so that $N_{dr}$ minimizes to $T_2:=(ha-d)(a-r)$ at $x_2=0, \ m=-1$.

Solving $T_1-T_2\geq 0$ gives $r\geq (ha-d)/2h$. Thus  $N_{dr}=T_2$
for $r\geq (ha-d)/2h$,  and $N_{dr}=T_1$ for otherwise. Obviously $0< (ha-d)/2h< a$.    Therefore we have
$$\begin{aligned}
\left\{
\begin{aligned}
\mathop{\max}\limits_{r\in \lbrace 0, 1, ..., s\rbrace} \lbrace N_{dr}\rbrace &=
\lfloor \dfrac{ha-d}{2h}\rfloor(ha+d)\ & \text{if}\  & s\leq (ha-d)/2h, \\
\mathop{\max}\limits_{r\in \lbrace s, ..., a-1\rbrace} \lbrace N_{dr}\rbrace &=\big(a-\lceil \dfrac{ha-d}{2h}\rceil\big)(ha-d)\ & \text{if}\  & s\geq (ha-d)/2h.
\end{aligned}
\right.
\end{aligned}$$
Finally by $g(x)=\max\{ N_{dr}\}-a$, we get the formula of $g(A)$ in this theorem.

For $n(A)$ and $s(A)$:

$\bullet$ If $r\leq (ha-d)/2h$,  we have
$N_{dr}=T_1=(ha+d)r$.

$\bullet$ If $r\geq(ha-d)/2h$,  we have
$N_{dr}=T_2=(ha-d)(a-r)$.

Let $s=\lfloor \frac{ha-d}{2h}\rfloor$.  Thus
\begin{small}
\begin{align*}
\sum_{r=1}^{a-1}N_{dr}&=\sum_{r=1}^{s}(ha+d)r+\sum_{r=s+1}^{a-1}(ha-d)(a-r)
\\&=\frac{ha+d}{2}s(s+1)+\frac{ha-d}{2}(a-s)(a-s-1),
\end{align*}
\end{small}
and
\begin{small}
\begin{align*}
\sum_{r=1}^{a-1}N_{dr}^2&=\sum_{r=1}^{s}(ha+d)^2r^2+\sum_{r=s+1}^{a-1}(ha-d)^2(a-r)^2
\\&=\frac{(ha+d)^2}{6}s(s+1)(2s+1)+\frac{(ha-d)^2}{6}(a-s-1)
(a-s)(2a-2s-1).
\end{align*}
\end{small}
By Lemma \ref{LiuXin001}, we have
\begin{small}
\begin{align*}
n(A)&=\frac{1}{a}\sum_{r=1}^{a-1}N_{dr}-\frac{a-1}{2}=\frac{ha+d}{2a}s(s+1)+\frac{ha-d}{2a}
(a-s)(a-s-1)-\frac{a-1}{2},
\end{align*}
\end{small}
and
\begin{small}
\begin{align*}
s(A)&=\frac{1}{2a}\sum_{r=1}^{a-1}N_{dr}^2-\frac{1}{2}\sum_{r=1}^{a-1}N_{dr}+\frac{a^2-1}{12}
\\&=\Big(\frac{(ha+d)^2(2s+1)}{12a}-\frac{ha+d}{4}\Big)s(s+1)
\\&+\Big(\frac{(ha-d)^2(2a-2s-1)}{12a}-\frac{ha-d}{4}\Big)
(a-s)(a-s-1)+\frac{a^2-1}{12}.
\end{align*}
\end{small}

Further more, we have
\begin{small}
\begin{align*}
\sum_{r=1}^{a-1}N_{dr}^p&=\sum_{r=1}^{s}(ha+d)^pr^p+\sum_{r=s+1}^{a-1}(ha-d)^p(a-r)^p
\\&=(ha+d)^p\sum_{r=1}^{s}r^p+(ha-d)^p\sum_{r=1}^{a-s-1}r^p.
\end{align*}
\end{small}
thus obtain a formula for
\begin{small}
\begin{align*}
&s_\mu(A)=\frac{1}{\mu+1}\sum_{\kappa=0}^{\mu}\binom{\mu+1}{\kappa}B_{\kappa}a^{\kappa-1}
\sum_{r=1}^{a-1}N_r^{\mu+1-\kappa}
+\frac{B_{\mu+1}}{\mu+1}(a^{\mu+1}-1)\,\\
&=\frac{1}{\mu+1}\sum_{\kappa=0}^{\mu}\binom{\mu+1}{\kappa}B_{\kappa}a^{\kappa-1}
\Big((ha+d)^{\mu+1-\kappa}\sum_{r=0}^{s}r^{\mu+1-\kappa}
+(ha-d)^{\mu+1-\kappa}\sum_{r=1}^{a-s-1}r^{\mu+1-\kappa}\Big)
\\&\ \ \ \ +\frac{B_{\mu+1}}{\mu+1}(a^{\mu+1}-1).
\end{align*}
\end{small}
This completes the proof.
\end{proof}

\section{Frobenius Formula for Almost Arithmetic Sequences}
In this section,  we will calculate the formula $g(A), n(A), s(A)$ of almost arithmetic sequences.
We give a short proof of the following theorem and then establish some new formulas using the same idea.
\begin{thm}[\cite{E. S. Selmer},\cite{T.Komatsu2022Arx}]\label{lsf}
Let $A=(a, ha+d, ha+2d, ..., ha+kd)$,  $(a, d)=1$,  $a, h, d, k\in \mathbb{P}$ and $1\leq k\leq a-1$. Suppose $a-1=sk+r_1$, $s\geq 0$ and $0\leq r_1\leq k-1$, then
\begin{small}
\begin{align*}
g(A)&=ha\Big(\lfloor\dfrac{a-2}{k}\rfloor+1\Big)+(d-1)(a-1)-1,\\
n(A)&=h\Big(\lfloor \frac{a-1}{k}\rfloor+1\Big)\Big(a-1-\frac{k}{2}\lfloor\frac{a-1}{k}\rfloor\Big)
+\frac{(d-1)(a-1)}{2},\\
s(A)&=\frac{a(s+1)h^2}{6}\big(ks^2+\frac{1}{2}ks+3sr_1+3r_1\big)+\frac{1}{6}(a-1)(d-1)
\big(ad-\frac{d}{2}-\frac{a}{2}-\frac{1}{2}\big)
\\&\ \ \ -\frac{h(s+1)}{4}\Big(-\frac{4}{3}dk^2s^2+\big(\frac{kd}{3}-4dr_1-d+a\big)ks+2r_1(-r_1d-d+a)\Big),\\
s_\mu(A)&=\frac{1}{\mu+1}\sum_{\kappa=0}^{\mu}\binom{\mu+1}{\kappa}B_{\kappa}a^{\kappa-1}
\sum_{r=1}^{a-1}\big(ha\lceil \frac{r}{k}\rceil+dr\big)^{\mu+1-\kappa}
+\frac{B_{\mu+1}}{\mu+1}(a^{\mu+1}-1).
\end{align*}
\end{small}
\end{thm}
\begin{proof}
By Lemma \ref{0202},  for a given $r$,  we have
\begin{align*}
N_{dr}=\min_{m\in\mathbb{N}} \{O_B(ma+r)\cdot ha+d\cdot (ma+r)\},
\end{align*}
where $O_B(M)=\min\{\sum_{i=1}^kx_i \mid \sum_{i=1}^ki\cdot x_i=M\}$.

If $M=s\cdot k+t$ with $1\leq t\leq k$,  then $O_B(M)=s+1$,  which
minimizes at  $x_k=s, x_t=1, x_i=0, (i\neq k, t)$ when $k\neq t$,
and minimizes at $x_k=s+1, x_i=0, (i\neq k)$ when $k=t$.

Now $N_{dr}(m)=(s+1)ha+d(ma+r)=ha\lceil  \frac{ma+r}{k} \rceil +d(ma+r) $ is increasing with respect to $m$. It follows that
$$N_{dr}=N_{dr}(0)=ha\lceil  \frac{r}{k} \rceil +dr.$$
Using this explicit formula, we can compute $g(A)$, $n(A)$ and $s(A)$ as follows.

For $g(A)$, notice $N_{dr}$ is increasing with respect to $r$. Hence
$$\mathop{\max}\limits_{r\in \lbrace 0, 1, ..., a-1\rbrace} \lbrace N_{dr}\rbrace =N_{d(a-1)}= ha\Big(\lfloor \dfrac{a-2}{k}\rfloor+1\Big)+d(a-1).$$

For $n(A)$, we have
\begin{align*}
\sum_{r=1}^{a-1}N_{dr}&=ha\sum_{r=1}^{a-1}\lceil\frac{r}{k}\rceil+\sum_{r=1}^{a-1}dr
\\&=ha((1+2+\cdots +s)k+(s+1)r_1)+d\frac{a(a-1)}{2}
\\&=ha\Big(\frac{s(s+1)k}{2}+(s+1)r_1\Big)+\frac{da(a-1)}{2}
\\&=ha\Big(\lfloor \frac{a-1}{k}\rfloor+1\Big)\Big(a-1-\frac{k}{2}\lfloor\frac{a-1}{k}\rfloor\Big)+\frac{da(a-1)}{2}.
\end{align*}
Applying Lemma \ref{LiuXin001} gives
\begin{align*}
n(A)=h\Big(\lfloor \frac{a-1}{k}\rfloor+1\Big)\Big(a-1-\frac{k}{2}\lfloor\frac{a-1}{k}\rfloor\Big)
+\frac{(d-1)(a-1)}{2}.
\end{align*}

For $s(A)$, we have
\begin{small}
\begin{align*}
\sum_{r=1}^{a-1}N_{dr}^2&=h^2a^2\sum_{r=1}^{a-1}(\lceil\frac{r}{k}\rceil)^2
+\sum_{r=1}^{a-1}d^2r^2+2had\sum_{r=1}^{a-1}r (\lceil\frac{r}{k}\rceil)
\\&=h^2a^2\Big(\sum_{m=1}^skm^2+(s+1)^2r_1\Big)+\frac{ad^2(a-1)(2a-1)}{6}
+2had\Bigg(\sum_{m=1}^s\sum_{r=(m-1)k+1}^{mk}mr+\sum_{r=sk+1}^{sk+r_1}(s+1)r\Bigg)
\\&={h}^{2}{a}^{2} \left( {\frac {s \left( s+1 \right)  \left( 2\,s+1
 \right) k}{6}}+ \left( s+1 \right) ^{2}{r_1} \right) +{\frac {a{d}
^{2} \left( a-1 \right)  \left( 2\,a-1 \right) }{6}}
\\&\ \ \ +2had\Big(-\frac{3k^2(s+1)^2}{4}+\frac{5(s+1)k^2}{12}+\frac{k(s+1)^2}{4}-\frac{(s+1)k^2}{4}
\\&\ \ \ +\frac{k^2(s+1)^3}{3}+\frac{r_1(s+1)(2sk+r_1+1)}{2}\Big).
\end{align*}
\end{small}
Applying Lemma \ref{LiuXin001} gives
\begin{small}
\begin{align*}
s(A)&=\frac{a(s+1)h^2}{6}\big(ks^2+\frac{1}{2}ks+3sr_1+3r_1\big)+\frac{1}{6}(a-1)(d-1)
\big(ad-\frac{d}{2}-\frac{a}{2}-\frac{1}{2}\big)
\\&-\frac{h(s+1)}{4}\Big(-\frac{4}{3}dk^2s^2+(\frac{kd}{3}-4dr_1-d+a)ks+2r_1(-r_1d-d+a)\Big).
\end{align*}
\end{small}

The formula for $s_\mu(A)$ is a direct consequence of Lemma \ref{LiuXin001}.
\end{proof}

Next we extend a recent result (corresponding to our $h=1$ case) of Takao Komatsu \cite{T.Komatsu22Arx}.
\begin{thm}
Let $A=(a, ha+(K+1)d, ha+(K+2)d, ..., ha+kd)$,  where $\gcd(a, d)=1$,  $a, h, d, K, k\in \mathbb{P}$, $K\leq \frac{k-1}{2}$ and $a\geq 2$. Let $a+K=qk+r_1$, $q\geq 0$ and $0\leq r_1\leq k-1$.
Then
$$\begin{aligned}
g(A)&=ha\cdot \lceil\frac{a+K}{k}\rceil+d(a+K)-a,\\
n(A)&=\left\{
\begin{aligned}
&h\Big((q+1)\big(\frac{qk}{2}+r_1-1\big)-K\Big)+\frac{(a-1)(d-1)}{2}+Kd\ & \text{if} &\ \ K+1\leq r_1\leq k, \\
&h\Big((q+1)\big(\frac{qk}{2}+r_1\big)-K-q\Big)+\frac{(a-1)(d-1)}{2}+Kd\ & \text{if} &\ \ 0\leq r_1\leq K. \\
\end{aligned}
\right.
\end{aligned}$$
\end{thm}
\begin{proof}
By Lemma \ref{0202},  for a given $r$,  we have
\begin{align*}
N_{dr}&=\min \{O_B(ma+r)\cdot ha+d(ma+r)\mid m\in\mathbb{N}\},
\end{align*}
where $O_B(M)=\min\big\{\sum_{i=1}^{k-K}x_i \mid \sum_{i=1}^{k-K}(K+i)\cdot x_i=M\big\}$.

If $M=ma+r=k\cdot s+t$,  where $s\geq 0, 1\leq t\leq k$.
When $K+1\leq t\leq k$,  we have $O_B(M)=s+1$,  which minimizes at $x_{k-K}=s,  x_{t-K}=1,  x_i=0, (i\neq k-K, t-K), k\neq t$ and $x_{k-K}=s+1, x_i=0, (i\neq k-K), k=t$.
When $1\leq t\leq K$, there is no solution for $s=0$. For $s\geq 1$,  we have $ma+r=(s-1)k+k+t$ and $O_B(M)=s+1$,  which minimizes at $x_{k-K}=s-1,  x_{k+t-2K-1}=1,  x_1=1,  x_i=0,  (i\neq 1, k+t-2K-1, k-K)$.
Therefore, we have
$$N_{dr}(m)=(s+1)ha+d(ma+r)=\lceil \frac{ma+r}{k}\rceil\cdot ha+d(ma+r),$$
which is increasing with respect to $m$. Therefore $N_{dr}=N_{dr}(0)$ for $K+1\leq r\leq a-1$ and $N_{dr}=N_{dr}(1)$ for $1\leq r\leq K$
(note that $N_{dr}(0)$ does not exit in this case). Thus we have
$$\begin{aligned}
N_{dr}=
\left\{
\begin{aligned}
&ha\lceil \frac{a+r}{k}\rceil+d(a+r)\ & \text{if} & \ 1\leq r\leq K, \\
&ha\lceil \frac{r}{k}\rceil+dr\ & \text{if} & \ K+1\leq r\leq a-1. \\
\end{aligned}
\right.
\end{aligned}$$

This explicit formula allow us to compute $g(A),n(A),s(A)$ easily. We only work out the $g(A)$ case, the other two cases
are omitted since the process is similar to that of Takao Komatsu.

Since $N_{dr}$ is increasing with respect to $r$, we have $\max\{N_{dr}\}=\max \{ N_{dK}, N_{d(a-1)}\} =N_{dK}$. This gives
\begin{align*}
g(A)=\max\{N_{dr}\}-a= \lceil\frac{a+K}{k}\rceil \cdot ha+d(a+K)-a.
\end{align*}
\end{proof}

Next,  we construct some interesting sequences and find their Frobenius formulas.
\begin{thm}\label{13579ss1}
Let $A=(a, ha+d, ha+3d, ha+5d, ..., ha+(2k+1)d)$,  where $(a, d)=1$,  $a, h, d, k\in \mathbb{P}$, $a>2$, $3\leq 2k+1\leq a-1$. And let $a-1=(2k+1)s+t$,  where $1\leq t\leq 2k+1$.
Then when $t\equiv 0\mod2$,  we have
$$g(A)=ha\Big(\lfloor \dfrac{a-2}{2k+1}\rfloor +2\Big)+(a-1)d-a.$$
When $t\equiv 1\mod2$,  we have
$$g(A)=\max \Big\{ ha\Big(\lfloor \dfrac{a-2}{2k+1}\rfloor +1\Big)+(a-1)d-a, \ ha\Big(\lfloor \dfrac{a-3}{2k+1}\rfloor +2\Big)+(a-2)d-a\Big\}.$$
Further more, we have
$$n(A)=hs\big(ks+t+\frac{1}{2}s-\frac{1}{2}\big)+(a-1)\big(\frac{d}{2}+h-\frac{1}{2}\big)
+h\lfloor\frac{t}{2}\rfloor.$$
\end{thm}
\begin{proof}
By Lemma \ref{0202},  for a given $r$,  we have
\begin{align*}
N_{dr}&=\min_{m\in\mathbb{N}} \{O_B(ma+r)\cdot ha+d(ma+r)\},
\end{align*}
where $O_B(M)=\min\{\sum_{i=0}^kx_i \mid x_0+\sum_{i=1}^k(2i+1)\cdot x_i=M\}$.
Suppose $M=(2k+1)\cdot s+t$,  where $1\leq t\leq 2k+1$.

$\bullet$ When $t$ is even,  we have $O_B(M)=s+2$,  which minimizes at  $x_k=s, x_{(t-2)/2}=1, x_0=1, x_i=0, (i\neq k, (t-2)/2, 0)$ for $t\neq 2$,  and minimizes at $x_k=s, x_0=2, x_i=0, (i\neq k, 0)$ for $t=2$. Then
$$N_{dr}(m)=ha(s+2)+d(ma+r)=ha\Big(\lfloor \dfrac{ma+r-1}{2k+1}\rfloor +2\Big)+d(ma+r).$$

$\bullet$ When $t$ is odd,  we have $O_B(M)=s+1$,  which minimizes at $x_k=s, x_{(t-1)/2}=1, x_i=0, (i\neq k, (t-1)/2)$ for $t\neq 2k+1$ and $t\neq 1$,  and minimizes at $x_k=s+1, x_i=0, (i\neq k)$ for $t=2k+1$,  and minimizes at $x_k=s, x_0=1, x_i=0, (i\neq k, 0)$ for $t=1$. Then
$$N_{dr}(m)=ha(s+1)+d(ma+r)=ha\Big(\lfloor \dfrac{ma+r-1}{2k+1}\rfloor +1\Big)+d(ma+r).$$

Obviously $N_{dr}(m+1)-N_{dr}(m)\geq ha(\lfloor\frac{a}{2k+1}\rfloor-1)+ da\geq 0$. In order to minimize $N_{dr}$,  we have
$$\begin{aligned}
N_{dr}=N_{dr}(0)=
\left\{
\begin{aligned}
&ha\Big(\lfloor \dfrac{r-1}{2k+1}\rfloor +2\Big)+dr \ \ & \text{if}\ \ & t\ \text{is even}, \\
&ha\Big(\lfloor \dfrac{r-1}{2k+1}\rfloor +1\Big)+dr \ \ & \text{if}\ \ & t\ \text{is odd}. \\
\end{aligned}
\right.
\end{aligned}$$

By observing the above formula,
if $a-1=(2k+1)\cdot s+t$ and $t$ is even,  we have
\begin{align*}
\mathop{\max}\lbrace N_{dr}\rbrace &=\max \Big\{ ha\Big(\lfloor \dfrac{a-2}{2k+1}\rfloor +2\Big)+(a-1)d, \ ha\Big(\lfloor \dfrac{a-3}{2k+1}\rfloor +1\Big)+(a-2)d\Big\}
\\&=ha\Big(\lfloor \dfrac{a-2}{2k+1}\rfloor +2\Big)+(a-1)d.
\end{align*}
If $a-1=(2k+1)\cdot s+t$ and $t$ is odd,  we have
\begin{align*}
\mathop{\max}\lbrace N_{dr}\rbrace &=\max \Big\{ ha\Big(\lfloor \dfrac{a-2}{2k+1}\rfloor +1\Big)+(a-1)d, \ ha\Big(\lfloor \dfrac{a-3}{2k+1}\rfloor +2\Big)+(a-2)d\Big\}.
\end{align*}
By $g(A)=\mathop{\max}\lbrace N_{dr}\rbrace-a$, the first part of the theorem is proved.

For $n(A)$, if $r-1=(2k+1)\widehat{s}+\widehat{t}-1$, $0\leq \widehat{t}-1\leq 2k$, thus

$\bullet$ If $\widehat{t}-1$ is odd,  we have
$N_{dr}=ha(\lfloor\frac{r-1}{2k+1}\rfloor+2)+dr$.

$\bullet$ If $\widehat{t}-1$ is even,  we have
$N_{dr}=ha(\lfloor\frac{r-1}{2k+1}\rfloor+1)+dr$.

When $r=a-1$, we have $a-2=(2k+1)s+t-1$, $0\leq t-1\leq 2k$, and
\begin{align*}
\sum_{r=1}^{a-1}N_{dr}&=\sum_{r=1}^{a-1}ha\Big(\lfloor\frac{r-1}{2k+1}\rfloor+1\Big)
+d\sum_{r=1}^{a-1}r+\sum_{r=1 \atop \widehat{t}-1=odd}^{a-1}ha
\\&=ha\big((1+2+\cdots +s-1\big)(2k+1)+st)+ha(a-1)+\frac{da(a-1)}{2}+ha\big(ks+\lfloor \frac{t}{2}\rfloor\big)
\\&=ha\Big((2k+1)\frac{s(s-1)}{2}+st\Big)+a(a-1)\big(\frac{d}{2}+h\big)+ha\big(ks+\lfloor \frac{t}{2}\rfloor\big).
\end{align*}
By Lemma \ref{LiuXin001}, we have
\begin{align*}
n(A)=hs\big(ks+t+\frac{1}{2}s-\frac{1}{2}\big)+(a-1)\big(\frac{d}{2}+h-\frac{1}{2}\big)+h\lfloor\frac{t}{2}\rfloor.
\end{align*}
\end{proof}

\begin{thm}\label{12468ss2}
Let $A=(a, ha+d, ha+2d, ha+4d, ..., ha+2kd)$,  where $(a, d)=1$,  $a, h, d, k\in \mathbb{P}$, $a>2$, $2\leq 2k\leq a-1$. Then

$\bullet$ when $a\equiv 0\mod2$,  if $a-1=2k\cdot s+t$, $1\leq t\leq 2k$,  we have
$$\begin{aligned}
g(A)=
\left\{
\begin{aligned}
& ha\Big(\lfloor \dfrac{a-2}{2k}\rfloor +2\Big)+(a-1)d-a\ \ & \text{if}\ \ & t\neq 1, \\
& ha\Big( \dfrac{a-2}{2k} +1\Big)+(a-1)d-a\ \ & \text{if}\ \ & t=1, \\
\end{aligned}
\right.
\end{aligned}$$

$\bullet$ when $a\equiv 1\mod2$, if $a-2=2k\cdot s+t$, $1\leq t\leq 2k$,  we have
$$\begin{aligned}
g(A)=
\left\{
\begin{aligned}
& \max\Big\{ha\Big(\lfloor \dfrac{a-2}{2k}\rfloor +1\Big)+(a-1)d-a, \ ha\Big(\lfloor \dfrac{a-3}{2k}\rfloor +2\Big)+(a-2)d-a\Big\}\ & \text{if}\ \ & t\neq 1, \\
& ha\Big(\lfloor \dfrac{a-2}{2k}\rfloor +1\Big)+(a-1)d-a \ & \text{if}\ \ & t=1. \\
\end{aligned}
\right.
\end{aligned}$$
Let $s=\lfloor \frac{a-2}{2k}\rfloor$, $t=a-1-2k\lfloor \frac{a-2}{2k}\rfloor$. We have
$$n(A)=h(s^2k+st-s-1)+(a-1)\big(h+\frac{d}{2}-\frac{1}{2}\big)+h\lceil \frac{t}{2}\rceil.$$
\end{thm}
\begin{proof}
By Lemma \ref{0202},  for a given $r$,  we have
\begin{align*}
N_{dr}&=\min_{m\in\mathbb{N}} \{O_B(ma+r)\cdot ha+d(ma+r)\},
\end{align*}
where $O_B(M)=\min\{\sum_{i=0}^kx_i \mid x_0+\sum_{i=1}^k2i\cdot x_i=M\}$.
Suppose $M=2k\cdot s+t$,  where $\ 1\leq t\leq 2k$.

$\bullet$ When $t$ is even,  we have $O_B(M)=s+1$,  which minimizes at $x_k=s, x_{t/2}=1, x_i=0, (i\neq k, t/2)$ for $t\neq 2k$,  and minimizes at $x_k=s+1, x_i=0, (i\neq k)$ for $t=2k$. This gives
$$N_{dr}(m)=ha(s+1)+d(ma+r)=ha\Big(\lfloor \dfrac{ma+r-1}{2k}\rfloor +1\Big)+d(ma+r).$$

$\bullet$ When $t$ is odd and $t\neq 1$,  we have $O_B(M)=s+2$,  which minimizes at $x_k=s, x_{(t-1)/2}=1, x_0=1, x_i=0, (i\neq 0, k, (t-1)/2)$. Then
$$N_{dr}(m)=ha(s+2)+d(ma+r)=ha\Big(\lfloor \dfrac{ma+r-1}{2k}\rfloor +2\Big)+d(ma+r).$$
When $t=1$,  we have $O_B(M)=s+1$,  which minimizes at $x_k=s, x_0=1, x_i=0, (i\neq 0, k)$. Then
$$N_{dr}(m)=ha(s+1)+d(ma+r)=ha\Big(\lfloor \dfrac{ma+r-1}{2k}\rfloor +1\Big)+d(ma+r).$$

Since $N_{dr}(m+1)-N_{dr}(m)\geq ha(\lfloor \frac{a}{2k}\rfloor-1)+da\geq 0$ ,  we have
$$\begin{aligned}
N_{dr}=N_{dr}(0)=
\left\{
\begin{aligned}
&ha\Big(\lfloor \dfrac{r-1}{2k}\rfloor +1\Big)+dr \ \ & \text{if}\ \ & t\ \text{is even or }\ t=1, \\
&ha\Big(\lfloor \dfrac{r-1}{2k}\rfloor +2\Big)+dr \ \ & \text{if}\ \ & t\ \text{is odd and}\  t\neq 1. \\
\end{aligned}
\right.
\end{aligned}$$

If $a$ is even,  then $r=a-1$ is odd and the corresponding $t$ is odd,  $r=a-2$ is even and the corresponding $t$ is even,  etc.  Observe the above formula,
if $a-1=2k\cdot s+t, \ t\neq 1$,  we have
\begin{align*}
\mathop{\max}\lbrace N_{dr}\rbrace &=\max \Big\{ ha\Big(\lfloor \dfrac{a-2}{2k}\rfloor +2\Big)+(a-1)d, \ ha\Big(\lfloor \dfrac{a-3}{2k}\rfloor +1\Big)+(a-2)d\Big\}
\\&=ha\Big(\lfloor \dfrac{a-2}{2k}\rfloor +2\Big)+(a-1)d.
\end{align*}
If $a-1=2k\cdot s+t, \ t= 1$,  we have
\begin{align*}
\mathop{\max}\lbrace N_{dr}\rbrace &=\max \Big\{ ha\Big(\lfloor \dfrac{a-2}{2k}\rfloor +1\Big)+(a-1)d, \ ha\Big(\lfloor \dfrac{a-4}{2k}\rfloor +2\Big)+(a-3)d\Big\}
\\&=ha\Big( \dfrac{a-2}{2k}+1\Big)+(a-1)d.
\end{align*}

If $a$ is odd,  we have $r=a-1$ is even and the corresponding $t$ is even,  $r=a-2$ is odd and the corresponding $t$ is odd,  etc.  Observe the above formula,
if $a-2=2k\cdot s+t, \ t\neq 1$,  we have
\begin{align*}
\mathop{\max}\lbrace N_{dr}\rbrace &=\max \Big\{ ha\Big(\lfloor \dfrac{a-2}{2k}\rfloor +1\Big)+(a-1)d, \ ha\Big(\lfloor \dfrac{a-3}{2k}\rfloor +2\Big)+(a-2)d\Big\}.
\end{align*}
If $a-2=2k\cdot s+t, \ t= 1$,  we have
\begin{align*}
\mathop{\max}\lbrace N_{dr}\rbrace &=\max \Big\{ ha\Big(\lfloor \dfrac{a-2}{2k}\rfloor +1\Big)+(a-1)d, \ ha\Big(\lfloor \dfrac{a-5}{2k}\rfloor +2\Big)+(a-4)d\Big\}
\\&=ha\Big(\lfloor \dfrac{a-2}{2k}\rfloor +1\Big)+(a-1)d.
\end{align*}

By $g(A)=\max\{N_{dr}\}-a$, the first part of the theorem is proved.

For $n(A)$, if $r-1=2k\cdot \widehat{s}+\widehat{t}-1$, $0\leq \widehat{t}-1\leq 2k-1$, thus

$\bullet$ If $\widehat{t}-1$ is odd,  we have
$N_{dr}=ha(\lfloor\frac{r-1}{2k}\rfloor+1)+dr$.

$\bullet$ If $\widehat{t}-1$ is even,  we have
$$\begin{aligned}
\left\{
\begin{aligned}
N_{dr}=ha\Big(\lfloor\frac{r-1}{2k}\rfloor+2\Big)+dr\ &\  \text{if}\  & \widehat{t}-1\neq 0, \\
N_{dr}=ha\Big(\lfloor\frac{r-1}{2k}\rfloor+1\Big)+dr\ &\  \text{if}\  & \widehat{t}-1= 0.
\end{aligned}
\right.
\end{aligned}$$

Therefore, when $r=a-1$, we have $a-2=2k\cdot s+t-1$, $1\leq t\leq 2k$ and
\begin{small}
\begin{align*}
\sum_{r=1}^{a-1}N_{dr}&=\sum_{r=1}^{a-1}\Big(ha(\lfloor \frac{r-1}{2k}\rfloor+1)+dr\Big)+\sum_{r=1, \widehat{t}-1=even\atop \widehat{t}-1\neq 0}^{a-1}ha
\\&=\sum_{r=1}^{a-1}ha\lfloor \frac{r-1}{2k}\rfloor+ha(a-1)+\frac{da(a-1)}{2}+\sum_{r=1, \widehat{t}-1=even\atop \widehat{t}-1\neq 0}^{a-1}ha
\\&=ha((1+2+\cdots+s-1)2k+st)+ha(a-1)+\frac{da(a-1)}{2}+ha\big((k-1)s+\lceil \frac{t}{2}\rceil-1\big)
\\&=ha(s(s-1)k+st)+ha(a-1)+\frac{da(a-1)}{2}+ha\big((k-1)s+\lceil \frac{t}{2}\rceil-1\big)
\\&=ha(s^2k+st-s-1)+a(a-1)(h+\frac{d}{2})+ha\lceil \frac{t}{2}\rceil,
\end{align*}
\end{small}
where $s=\lfloor \frac{a-2}{2k}\rfloor$, $t=a-1-2k\lfloor \frac{a-2}{2k}\rfloor$. By Lemma \ref{LiuXin001}, we have
\begin{align*}
n(A)=h(s^2k+st-s-1)+(a-1)\Big(h+\frac{d}{2}-\frac{1}{2}\Big)+h\lceil \frac{t}{2}\rceil.
\end{align*}
\end{proof}

The above results can be verified by Mathematica. First we use Mathematica to calculate $g(A)$ , then we calculate $n(A)$ and $s(A)$. In Mathematica, we have
\begin{enumerate}
  \item Input: FrobeniusNumber$[\{5, 16, 19, 22\}]$

  \item Output:33

  \item Input: Table[FrobeniusSolve$[{5, 16, 19, 22}, i], \{i, 33\}$]

  \item Output:$\{\{\}, \{\}, \{\}, \{\}, \{\{1, 0, 0, 0\}\}, \{\}, \{\}, \{\}, \{\}, \{\{2, 0, 0,0\}\}, \{\}, \{\}, \{\}, \{\}, \\
      \{\{3, 0, 0, 0\}\},\{\{0, 1, 0, 0\}\}, \{\}, \{\}, \{\{0,0, 1, 0\}\}, \{\{4, 0, 0, 0\}\}, \{\{1, 1, 0, 0\}\},\\
      \{\{0, 0, 0, 1\}\}, \{\}, \{\{1,0, 1, 0\}\}, \{\{5, 0, 0, 0\}\}, \{\{2, 1, 0, 0\}\}, \{\{1, 0, 0,1\}\}, \{\}, \\
      \{\{2, 0, 1, 0\}\}, \{\{6, 0, 0, 0\}\}, \{\{3, 1, 0, 0\}\}, \{\{0, 2, 0,0\}, \{2, 0, 0, 1\}\}, \{\}\}$
\end{enumerate}

For example, the penultimate result above indicates that when $i=32$, there are two non-negative integer solutions to $5x_1+16x_2+19x_3+22x_4=32$, i.e., $(x_1,x_2,x_3,x_4)=(0,2,0,0), (x_1,x_2,x_3,x_4)=(2,0,0,1)$. So for $A=(5,16,19,22)$, we have $g(A)=33, n(A)=17$ and $s(A)=209$.

\section{Another Idea: Extract Constant Term}

In this section, we introduce ``Constant Term Method" to calculate $\sum_{r=0}^{a-1}N_r$, $\sum_{r=0}^{a-1}N_r^2$, $\sum_{r=0}^{a-1}N_r^3$ etc.
The idea can be carried out by Maple.

\subsection{Introduction to Constant Term Method}
In \cite{Xin15}, Xin gave a polynomial time algorithm for MacMahon's partition analysis in a suitable condition, and develop a Euclid style algorithm with an implementation by the Maple package CTEuclid. For a more detailed description, see \cite{Xin04,Xin15}, our ideas mainly come from these two articles.
The new point that the ideas in \cite[Section 3]{Xin15} can be adapted to solve problems involving symbolic integer parameters. Such problems arise naturally in this article, but cannot be handled by existing packages.

Suppose $E(x)$ is written as a sum of \emph{simple Elliott-rational functions} in the following form:
\begin{align*}
E(x)=\sum_{i}\frac{L_i(x)}{(1-x^{b_{i1}})(1-x^{b_{i2}})\cdots (1-x^{b_{im}})},
\end{align*}
where the $L_i(x)$'s are Laurent polynomials, and $b_{ij}, 1\leq j\leq m$ are integers. Under the premise that $E(1)$ exists, our purpose is to calculate $E(1)$.

In \cite{Xin15}, two methods are introduced. One is by computing $\CT_t E(1+t)$, and the other is by computing $\CT_t E(e^t)$, where $\CT_t F(t)$ denotes taking constant term in the Laurent series expansion of $F(t)$. We adopt the exponential substitution $x\mapsto e^t$, which leads to
\begin{align*}
E(e^t)=\sum_{i}\frac{L_i(e^t)}{(1-e^{b_{i1}t})(1-e^{b_{i2}t})\cdots (1-e^{b_{im}t})}.
\end{align*}
Regards both sides of the above equation as formal Laurent series in $t$. Then by the linearity of the constant term operator $\mathrm{CT}_t$, we have
\begin{prop}
Suppose $E(1)$ exists, we have
\begin{align*}
E(1)=E(e^0)=\underset{t}{\mathrm{CT}}E(e^t)=\sum_{i}\underset{t}{\mathrm{CT}}\frac{L_i(e^t)}{(1-e^{b_{i1}t})(1-e^{b_{i2}t})\cdots (1-e^{b_{im}t})}.
\end{align*}
\end{prop}
For fix $m$, we only need the following two formulas and do polynomial multiplications:
\begin{equation}\label{tet}
\frac{t}{1-e^t}=\sum_{n=0}^m-\frac{\mathcal{B}_n}{n!}t^n+o(t^m)
=-1+\frac{1}{2}t-\frac{1}{12}t^2+\frac{1}{720}t^4+\cdots+o(t^m),
\end{equation}
\begin{equation}\label{et}
e^t=\sum_{n=0}^m\frac{1}{n!}t^n+o(t^m)=1+t+\frac{1}{2}t^2+\frac{1}{6}t^3+\cdots+ o(t^m).
\end{equation}
The $\mathcal{B}_n$ are the well-konwn Bernoulli numbers.

\begin{thm}\label{ctpolynomE}
Let $L(x)$ be a Laurent polynomial with $U$ terms, and $b_j, 1\leq j\leq m$ are integers. The following formula can be quickly calculated in finite steps:
\begin{align*}
\underset{t}{\mathrm{CT}}\frac{L(e^t)}{(1-e^{b_{1}t})(1-e^{b_{2}t})\cdots (1-e^{b_{m}t})}.
\end{align*}
\end{thm}
\begin{proof}
By Equation \eqref{tet}, we can use $m-1$ multiplications to obtain
\begin{align*}
\frac{t^m}{\prod_{j=1}^m(1-e^{b_{j}t})}=\prod_{j=1}^m\sum_{n\geq 0}-\frac{\mathcal{B}_nb_j^{n-1}}{n!}t^n=\sum_{n=0}^m c_n t^n+o(t^m).
\end{align*}
This is a power series in $t$. By Equation \eqref{et}, we have the expansion
$$L(e^t)=\sum_{n=0}^m d_n t^n+o(t^m).$$
Therefore we have
\begin{align*}
\underset{t}{\mathrm{CT}}\frac{L(e^t)}{\prod_{j=1}^m(1-e^{b_{j}t})}
=[t^m]\Big(\sum_{n=0}^m d_n t^n\cdot \sum_{n=0}^m c_n t^n\Big)=\sum_{n=0}^m d_n c_{m-n}.
\end{align*}
\end{proof}

Observe that the above process works for symbolic $b_{ij}$. It is not hard to rewrite the code in Maple to carry out the above process. We store
Equations \eqref{tet} and \eqref{et} for $m=30$ in advance. This is sufficient for problems with $m\leq 30$. Maple can help us quickly calculate $n(A),s(A)$, as we shall discuss next.

\subsection{Calculating $n(A)$ and $s(A)$}
In \cite{Brown1993}, Brown and Shiue use the following ideas to calculate $s(a,b)$. Let $A$ be any set of positive integers with $\gcd(A)=1$. Set $f_{\mathcal{A}}(x)=\sum_{a\in \mathcal{A}}x^a$. Then
$$f_{{\mathcal{NR}}(A)}(x)=\frac{1}{1-x}-f_{{\mathcal{R}}(A)}(x),$$
and
$$g(A)=\deg f_{{\mathcal{NR}}(A)}(x),\ \ n(A)=\lim_{x\rightarrow 1}f_{{\mathcal{NR}}(A)}(x)\ \ s(A)=\lim_{x\rightarrow 1}f_{{\mathcal{NR}}(A)}^{\prime}(x).$$
In \cite{M.Koppe}, M. K\"oppe, S. Verdoolaege and K. M. Woods implement the Barvinok--Woods integer projection algorithm to write
$f(x)$ (below) as a short sum of simple rational functions. We also have a decomposition of $f(x)$, but our simple rational functions are easily seen to be a polynomial with nonnegative integer coefficients.

Let $f(x):=\sum_{r=0}^{a-1}x^{N_r}$. Then $g(A)=\deg f(x)-a$, and $f(x)$ is a polynomial represented as a sum of rational functions. Let $n,p \in \mathbb{P}$ and $(n)_p=n(n-1)\cdots (n-p+1)$ be the following factorial. For $k\in \mathbb{N}$, Stirling number of the first kind $s(n, k)$ and the second kind $S(n, k)$ have the following relations:
$$(n)_p= \sum_{k=0}^ps(p,k)n^k,\ \ \ n^p=\sum_{k=0}^{p}S(p,k)(n)_k.$$
The relevant conclusions can be obtained in \cite{RP. Stanley} for Stirling number of the first kind $s(n, k)$ and the second kind $S(n, k)$. By taking the $p$-th order differential of $f(x)$, i.e.,
$$f^{(p)}(x)=\big(\frac{d}{dx}\big)^p f(x)=\sum_{r=1}^{a-1}(N_r)_p x^{N_r-p},$$
we have $f^{(p)}(1)=\sum_{r=1}^{a-1}(N_r)_p$ and
$$\sum_{r=1}^{a-1}N_r^p=\sum_{r=1}^{a-1}\sum_{k=0}^p S(p,k) (N_r)_k= \sum_{k=0}^p S(p,k)f^{(k)}(1).$$

Therefore, we only need to solve $f^{\prime}(1), f^{\prime\prime}(1)$, $f^{(\mu+1)}(1)$ and further we can get $n(A), s(A), s_{\mu}(A)$. We can easily obtain the following Theorem.
\begin{thm}
Let $f(x):=\sum_{r=0}^{a-1}x^{N_r}$. Then the Frobenius number, Sylvester number and Sylvester sum are respectively:
\begin{align*}
g(A)&=\max \mathcal{NR}=\deg f(x) -a,\\
n(A)&=\sum_{n\in \mathcal{NR}}1=\frac{1}{a}f^{\prime}(1)-\frac{a-1}{2},\\
s(A)&=\sum_{n\in \mathcal{NR}}n=\frac{1}{2a}f^{\prime\prime}(1)-\frac{a-1}{2a}f^{\prime}(1)+\frac{a^2-1}{2}.
\end{align*}
\end{thm}

In \cite{T.Komatsu2021,T.Komatsu2022}, T. Komatsu introduce another statistic {\it Sylvester weighted power sum}, that is
$$
s_\mu^{(\lambda)}(A):=\sum_{n\in{\mathcal{NR}}(A)}\lambda^n n^\mu
$$
for a positive integer $\mu$ with weight $\lambda(\ne 0,1)$, $\lambda\in \mathbb{C}$. Weighted sums include the so-called alternate sums \cite{W.Wang2008} in special cases, i.e., $\lambda=-1$.

Now, let $g(x)=\sum_{n\in \mathcal{NR}}x^n$. Note: Do not confuse $g(x)$ and $g(A)$. We have
$$g(x)=\frac{1}{1-x}-\frac{1}{1-x^a}\sum_{r=0}^{a-1} x^{N_r}=\frac{1}{1-x}-\frac{1}{1-x^a}f(x).$$
Further we have
$$g^{\prime}(x)=\frac{1}{(1-x)^2}-\frac{f^{\prime}(x)}{(1-x^a)}-\frac{af(x)\cdot x^{a-1}}{(1-x^a)^2}.$$

\begin{thm}\label{gnsssmula}
Let $g(x)=\sum_{n\in \mathcal{NR}}x^n$. Then the Frobenius number, Sylvester number, Sylvester sum, Sylvester power sum, and Sylvester weighted power sum are respectively:
\begin{align*}
&g(A)=\max \mathcal{NR}=\deg g(x),\\
&n(A)=\sum_{n\in \mathcal{NR}}1=g(1)=\underset{t}{\mathrm{CT}}g(e^t),\\
&s(A)=\sum_{n\in \mathcal{NR}}n=g^{\prime}(1)=\underset{t}{\mathrm{CT}}g^{\prime}(e^t),\\
&s_{\mu}(A)=\sum_{n\in \mathcal{NR}}n^{\mu}=\sum_{k=0}^{\mu}S(\mu, k)g^{(k)}(1)=\sum_{k=0}^{\mu}S(\mu, k)\underset{t}{\mathrm{CT}}g^{(k)}(e^t),\\
&s_\mu^{(\lambda)}(A)=\sum_{n\in{\mathcal{NR}}}\lambda^n n^\mu=\sum_{k=0}^{\mu} S(\mu,k) g^{(k)}(\lambda),
\end{align*}
where $\mu$ is a positive integer, $\lambda(\ne 0,1)$, $\lambda\in \mathbb{C}$ and
$$g^{(k)}(x)=\frac{k!}{(1-x)^{k+1}}-\sum_{i=0}^k \binom{k}{i}\Big(\frac{1}{1-x^a}\Big)^{(i)}f^{(k-i)}(x).$$
\end{thm}
\begin{proof}
From the definition of $g(x)$, we can easily get the formulas of $g(A)$, $n(A)$ and $s(A)$. By
$$g^{(\mu)}(x)=\sum_{n\in \mathcal{NR}}(n)_{\mu} x^{n-\mu},$$
we have $g^{(\mu)}(1)=\sum_{n\in \mathcal{NR}}(n)_{\mu}$ and
$$\sum_{n\in \mathcal{NR}}n^{\mu}=\sum_{n\in \mathcal{NR}}\sum_{k=0}^{\mu} S(\mu,k)(n)_k= \sum_{k=0}^{\mu} S(\mu,k)g^{(k)}(1).$$
So we get the formula of $s_{\mu}(A)$.

Now, let $\lambda(\ne 0,1)$, $\lambda\in \mathbb{C}$. We consider the $g(\lambda x)=\sum_{n\in \mathcal{NR}}\lambda^n x^n$.
By
$$g^{(\mu)}(\lambda x)=\sum_{n\in \mathcal{NR}}\lambda^n (n)_{\mu}x^{n-\mu},$$
we have $g^{(\mu)}(\lambda)=\sum_{n\in \mathcal{NR}}\lambda^n (n)_{\mu}$ and
$$\sum_{n\in \mathcal{NR}}\lambda^n n^{\mu}=\sum_{n\in \mathcal{NR}}\lambda^n \sum_{k=0}^{\mu} S(\mu,k)(n)_k =\sum_{k=0}^{\mu} S(\mu,k) g^{(k)}(\lambda).$$
This completes the proof.
\end{proof}

We find that the formula of $s_{\mu}(A)$ is slightly complicated in Theorem \ref{gnsssmula}. Therefore, the following definition may be more natural. Define $\widehat{s}_{\mu}(A)=\sum_{n\in \mathcal{NR}}\binom{n}{\mu}$ to be {\it Sylvester binomial moment}. Then we have
\begin{align*}
\widehat{s}_{\mu}(A)=\sum_{n\in \mathcal{NR}}\binom{n}{\mu}=\frac{1}{\mu !}g^{(\mu)}(1)=\frac{1}{\mu !}\underset{t}{\mathrm{CT}}g^{(\mu)}(e^t).
\end{align*}

\begin{exa}
We consider the arithmetic sequence $A=(a,ha+d,ha+2d,..., ha+kd)$, $1\leq k\leq a-1$. By $N_{dr}=ha\lceil \frac{r}{k}\rceil+dr$ and $a-1=sk+r_1, s\geq 0, 0\leq r_1\leq k-1$, we have
\begin{small}
\begin{align*}
f(x)&=\sum_{r=0}^{a-1}x^{N_{dr}}
\\&=1+\frac{1-x^{dk}}{1-x^d}(x^{ha+d}+x^{2ha+d(k+1)}+\cdots +x^{sha+d((s-1)k+1)})+x^{(s+1)ha+d(sk+1)}\frac{1-x^{dr_1}}{1-x^d}
\\&=1+\frac{1-x^{dk}}{1-x^d}\cdot \frac{x^{ha+d}(1-x^{(ha+dk)s})}{1-x^{ha+dk}}
+x^{(s+1)ha+d(sk+1)}\frac{1-x^{dr_1}}{1-x^d}.
\end{align*}
\end{small}
Maple can do derivative quickly, or one can compute by hand, to obtain
\begin{small}
\begin{align*}
f^{\prime}(x)=&-{\frac {{x}^{dk}dk{x}^{ah+d} \left( 1-{x}^{ \left( ah+dk \right) s}
 \right) }{x \left( 1-{x}^{d} \right)  \left( 1-{x}^{ah+dk} \right) }}
+{\frac { \left( 1-{x}^{dk} \right) {x}^{ah+d} \left( 1-{x}^{ \left( a
h+dk \right) s} \right) {x}^{d}d}{ \left( 1-{x}^{d} \right) ^{2}
 \left( 1-{x}^{ah+dk} \right) x}}
\\&+{\frac { \left( 1-{x}^{dk} \right) {
x}^{ah+d} \left( ah+d \right)  \left( 1-{x}^{ \left( ah+dk \right) s}
 \right) }{x \left( 1-{x}^{d} \right)  \left( 1-{x}^{ah+dk} \right) }}
-{\frac { \left( 1-{x}^{dk} \right) {x}^{ah+d}{x}^{ \left( ah+dk
 \right) s} \left( ah+dk \right) s}{x \left( 1-{x}^{d} \right)
 \left( 1-{x}^{ah+dk} \right) }}
\\&+{\frac { \left( 1-{x}^{dk} \right) {x
}^{ah+d} \left( 1-{x}^{ \left( ah+dk \right) s} \right) {x}^{ah+dk}
 \left( ah+dk \right) }{ \left( 1-{x}^{d} \right)  \left( 1-{x}^{ah+dk
} \right) ^{2}x}}
\\&+{\frac {{x}^{h \left( s+1 \right) a+d \left( ks+1
 \right) } \left( h \left( s+1 \right) a+d \left( ks+1 \right)
 \right)  \left( 1-{x}^{d{r_1}} \right) }{x \left( 1-{x}^{d}
 \right) }}
\\&-{\frac {{x}^{h \left( s+1 \right) a+d \left( ks+1 \right)
}{x}^{d{r_1}}d{r_1}}{x \left( 1-{x}^{d} \right) }}+{\frac {{x}^{
h \left( s+1 \right) a+d \left( ks+1 \right) } \left( 1-{x}^{d{r_1}
} \right) {x}^{d}d}{ \left( 1-{x}^{d} \right) ^{2}x}}
\\:=&f_1+f_2+f_3+f_4+f_5+f_6+f_7+f_8.
\end{align*}
\end{small}
Now, we have $\CT_t f^{\prime}(e^t)=\sum_{i=1}^8 \CT_t f_i(e^t)$. Let's compute the first term as an example.
The simplest way using Maple is to take series expansion of $f_1(t)$ and then take the constant term. This works fine for this problem,
but becomes slow when the \emph{pole order} $m$ at $t=0$ is big. Here we mean the Laurent series of $f_1(t)$ is of the form
$\sum_{j=m}^\infty a_j t^j$ with $a_m\neq 0$.
We explain how to use Theorem \ref{ctpolynomE} to do the computation.

Observe that $t f_1(e^t)=\ell_0 +\ell_1 t+o(t)$ is a power series, and our task is to compute $\ell_1$.
The pole order of $f_1(e^t)$ is 1.
Thus we can write
\begin{small}
\begin{align*}
tf_1(e^t)&=- dk \left( {{e}^{t}}
 \right) ^{dk-1+ah+d} \cdot \frac{1- \left( {{e}^{t}} \right) ^{ \left( ah+d
k \right) s}}{t}  \cdot \frac{t} { 1- \left( {{e}^{t}}
 \right) ^{d}} \cdot \frac{t}{ 1- \left( {{e}^{t}} \right) ^{ah+dk}
 }\\
 &=- dk(1+(dk-1+ah+d)t) \cdot (ah+dk)s(-1-\frac12(ah+dk)st)\\
 &\quad\quad \cdot \frac1d(-1+\frac{d}{2}t) \cdot \frac1{ah+dk}(-1+\frac12(ah+dk)t)
 ,
\end{align*}
\end{small}
where we have omitted the $o(t)$ terms for each factor (separated by $\cdot$). We obtain
\begin{align*}
f_1(1)=f_1(e^0)=\underset{t}{\mathrm{CT}}f_1(e^t)=  \frac{ks}{2}(ahs+dks+ah+dk+d-2).
\end{align*}
The above steps can be carried out by a Maple procedure.

Similarly, we can solve $\underset{t}{\mathrm{CT}}f_i(e^t)$ for all $i$, and obtain
$$f^{\prime}(1)=\underset{t}{\mathrm{CT}}f^{\prime}(e^t)
=\sum_{i=1}^8\underset{t}{\mathrm{CT}}f_i(e^t)=\frac{a(a-1)d}{2}+ha(s+1)(a-1-\frac{ks}{2}).$$
Then we can get $n(A)$. The result agrees with that in Theorem \ref{lsf}.

In the same way, we can calculate the higher order derivatives to get $s(A)$ and $s_{\mu}(A)$.
\end{exa}

By programming the procedure of ``Extract Constant Term" above in Maple, once we get $f(x)$, we can quickly obtain $n(A)$ and $s(A)$, and even more formulas. In Appendix, we only give $f(x)$ for the previous theorems. We calculate their $n(A)$ and $s(A)$ by the Maple, which agree with known results. Furthermore, we obtain the following result, which seems too complicated for direct computation using the explicit formula of $N_r$.

\begin{thm}
Let $A=(a, ha+d, ha+2d, ha+jd)$,  $a, j, h, d\in \mathbb{P}, h\geq d$. The restrictions are the same as in Theorem \ref{012j}. Therefore, we have
\begin{align*}
s(A)&=\frac{1}{24a}\Big({a}^{2}{h}^{2}{j}^{3}k+3{a}^{2}{h}^{2}{j}^{2}{k}^{2}+4{a}^{2}{h}^{
2}j{k}^{3}+3adh{j}^{3}{k}^{2}+8adh{j}^{2}{k}^{3}+4{d}^{2}{j}^{3}
{k}^{3}-3{j}^{2}{a}^{2}{h}^{2}k
\\& -3{a}^{2}{h}^{2}{j}^{2}t-6{a}^{2}
{h}^{2}j{k}^{2}-12{a}^{2}{h}^{2}jkt+3{a}^{2}{h}^{2}j{t}^{2}-12{a
}^{2}{h}^{2}{k}^{2}t+6{a}^{2}{h}^{2}k{t}^{2}-{a}^{2}{h}^{2}{t}^{3}
\\& +adh{j}^{3}k-6adh{j}^{2}{k}^{2}-12adh{j}^{2}kt-24adhj{k}^{2}t+6a
dhjk{t}^{2}-12{d}^{2}{j}^{2}{k}^{2}t
\\&+12{a}^{2}{h}^{2}jt+24{a}^{2}{h}^{2}kt-6{a}^{2}{h}^{2}{t}^{2}-3{a}^{2}h{j}^{2}k-6{a}^{2}hj{k
}^{2}-5adh{j}^{2}k+24adhjkt+6adhj{t}^{2}
\\&+12adhk{t}^{2}-4adh{t}^{3}-6ad{j}^{2}{k}^{2}-6{d}^{2}{j}^{2}{k}^{2}+12{d}^{2}jk{t}^{
2}+6{a}^{2}hjk+6{a}^{2}hjt+12{a}^{2}hkt
\\& -3{a}^{2}h{t}^{2}+6adhjt+12adhkt-15adh{t}^{2}+12adjkt+12{d}^{2}jkt-4{d}^{2}{t}^{3
}-12{a}^{2}ht+6adjk
\\&-6ad{t}^{2} +2{d}^{2}jk-6{d}^{2}{t}^{2}+2{a}^{3}-6adt-2{d}^{2}t-2a
+\widetilde{s}(A)\Big),
\end{align*}
where, if $j-1$ is even and $t$ is even,
$$\widetilde{s}(A)=j{a}^{2}{h}^{2}k-3{a}^{2}{h}^{2}{k}^{2}-9adhj{k}^{2}+3{a}^{2}{h}^{2}k
-11{a}^{2}{h}^{2}t+5adhjk+3{a}^{2}hk+3adhk-8adht,$$
if $j-1$ is even and $t$ is odd,
\begin{align*}
\widetilde{s}(A)&=j{a}^{2}{h}^{2}k-3{a}^{2}{h}^{2}{k}^{2}-9adhj{k}^{2}+3{a}^{2}{h}^{2}j
+9{a}^{2}{h}^{2}k-14{a}^{2}{h}^{2}t+11adhjk
\\&\ \ \ -6{a}^{2}{h}^{2}+3{a}^{2}hk
+3adhk-14adht-3{a}^{2}h-3adh,
\end{align*}
if $j-1$ is odd and $t$ is odd,
\begin{align*}
\widetilde{s}(A)&=4j{a}^{2}{h}^{2}k-6adhj{k}^{2}-3{a}^{2}{h}^{2}j-6{a}^{2}{h}^{2
}k-11{a}^{2}{h}^{2}t+2adhjk+6{a}^{2}{h}^{2}
\\&\ \ \ -8adht+3{a}^{2}h+3adh,
\end{align*}
if $j-1$ is odd and $t$ is even,
\begin{align*}
\widetilde{s}(A)=4j{a}^{2}{h}^{2}k-6adhj{k}^{2}-14{a}^{2}{h}^{2}t+8adhjk-14adht.
\end{align*}
\end{thm}

\section{Future Projects}
We established a combinatorial approach to Frobenius numbers of some special sequences of the form $A=(a,ha+dB)$.
If the optimization problem $O_B(M)$ is easy to solve, then we can obtain explicit formula of $N_r$, and hence possibly solve the Frobenius problem.

This approach succeeds for many special sequences, and we obtained many formulas that had not appeared before.
If we are lucky, we can write $f(x)$ as a sum of simple rational functions,
from which we can easily deduce $g(A)$, and compute $n(A)$, $s(A)$, $s_{\mu}(A)$ and $s_{\mu}^{(\lambda)}(A)$ by ``Extracting constant terms".

Our next project is to study the case $A=(a,ha+dB)$ when $O_B(M)$ can be solved by the greedy algorithm.
We hope to find a way to efficiently write the corresponding $f(x)$ as a short positive sum of simple rational functions.
The project is supported by the following two facts. There is a large class of $B$ such that $O_B(M)$. See \cite{AnnAdamaszek}.
Directly analyse from $N_r$ may become pretty hard as can be seen from the case $B=(1,2,j)$ in subsection \ref{hard12jj}.
However, the corresponding $f(x)$ is not bad.

One of our future projects is to generalize our combinatorial approach to case $A=(a, h_1a+b_1, h_2a+b_2,\dots, h_ka+b_k)$. Then we need to solve a weighted version of the $O_B(M)$ problem:
$$O_{B}(M):=\min\Big\{\sum_{i=1}^kh_ix_i \mid \sum_{i=1}^k b_ix_i=M, \ x_i\in\mathbb{N}, 1\leq i\leq k\Big\}.$$
This consideration is natural from the denumerant side by the CTEuclid algorithm.

\appendix
\section{the formula of $f(x)$}

\begin{prop}
Let $A=(a, ha+d, ha+jd)$,  $a, j>2$,  and $a=kj-t,  k\geq 1,  0\leq t\leq j-1, gcd(a,d)=1, d\leq h$. Suppsoe $hk+d-ht\geq 0$. Then
\begin{align*}
f(x)=\frac{1-x^{j(ha+d)}}{1-x^{ha+d}}\cdot \frac{(1-x^{(ha+jd)(k-1)})}{1-x^{ha+jd}}
+\frac{x^{(k-1)(ha+jd)}(1-x^{(ha+d)(j-t)})}{1-x^{ha+d}}.
\end{align*}
\end{prop}

\begin{prop}\label{aha12jdd}
Let $A=(a, ha+d, ha+2d, ha+jd)$,  $h\geq d$, $a, j\in \mathbb{N}$,  $a\geq 2$ and $j\geq 4$. Moreover let $a=kj-t$,  where $k\geq 1$ and $0\leq t\leq j-1$,  and we require $k+1-\lceil \frac{t}{2}\rceil\geq 0$.

If $j-1$ is even and $t$ is even, we have
\begin{align*}
f(x)=&\frac{1-x^{(ha+jd)k}}{1-x^{ha+jd}}
+\frac{(1-x^{\frac{(ha+2d)(j-1)}{2}})(1-x^{(ha+jd)(k-1)})(x^{ha+d}+x^{ha+2d})}{(1-x^{ha+2d})(1-x^{ha+jd})}
\\&+\frac{(1-x^{\frac{(ha+2d)(j-t-1)}{2}})(x^{kha+((k-1)j+1)d}+x^{kha+((k-1)j+2)d})}{1-x^{ha+2d}}.
\end{align*}

If $j-1$ is even and $t$ is odd, we have
\begin{align*}
f(x)=&\frac{1-x^{(ha+jd)k}}{1-x^{ha+jd}}
+\frac{(1-x^{\frac{(ha+2d)(j-1)}{2}})(1-x^{(ha+jd)(k-1)})(x^{ha+d}+x^{ha+2d})}{(1-x^{ha+2d})(1-x^{ha+jd})}
\\&+\frac{x^{kha+((k-1)j+1)d}(1-x^{\frac{(ha+2d)(j-t)}{2}})}{1-x^{ha+2d}}
+\frac{x^{kha+((k-1)j+2)d}(1-x^{\frac{(ha+2d)(j-t-2)}{2}})}{1-x^{ha+2d}}.
\end{align*}

If $j-1$ is odd and $t$ is odd, we have
\begin{align*}
f(x)=&\frac{1-x^{(ha+jd)k}}{1-x^{ha+jd}}
+\frac{(1-x^{(ha+jd)(k-1)})(x^{ha+d}(1-x^{\frac{(ha+2d)j}{2}})
+x^{ha+2d}(1-x^{\frac{(ha+2d)(j-2)}{2}}))}{(1-x^{ha+2d})(1-x^{ha+jd})}
\\&+\frac{(1-x^{\frac{(ha+2d)(j-t-1)}{2}})(x^{kha+((k-1)j+1)d}+x^{kha+((k-1)j+2)d})}{1-x^{ha+2d}}.
\end{align*}

If $j-1$ is odd and $t$ is even, we have
\begin{align*}
f(x)=&\frac{1-x^{(ha+jd)k}}{1-x^{ha+jd}}
+\frac{(1-x^{(ha+jd)(k-1)})(x^{ha+d}(1-x^{\frac{(ha+2d)j}{2}})
+x^{ha+2d}(1-x^{\frac{(ha+2d)(j-2)}{2}}))}{(1-x^{ha+2d})(1-x^{ha+jd})}
\\&+\frac{x^{kha+((k-1)j+1)d}(1-x^{\frac{(ha+2d)(j-t)}{2}})}{1-x^{ha+2d}}
+\frac{x^{kha+((k-1)j+2)d}(1-x^{\frac{(ha+2d)(j-t-2)}{2}})}{1-x^{ha+2d}}.
\end{align*}
\end{prop}

\begin{prop}
Let $A=(a, ha-d, ha+d)$,  $d, h\in \mathbb{P}, \gcd(a, d)=1$, $ha-d>1$ and
$r_1=\lfloor \frac{ha-d}{2h}\rfloor$. Then
\begin{align*}
f(x)=\frac{1-x^{(ha+d)(r_1+1)}}{1-x^{ha+d}}
+\frac{x^{ha-d}(1-x^{(ha-d)(a-r_1-1)})}{1-x^{ha-d}}.
\end{align*}
\end{prop}

\begin{prop}
Let $A=(a^2, ha^2+d, ha^2+ad, ha^2+(a+1)d)$,  $a>1, gcd(a,d)=1$. We have
\begin{align*}
f(x)&=\sum_{m=0}^{a-1}\sum_{i=0}^{a-m-1}x^{(m+i)ha^2+(m(a+1)+i)d}
+\sum_{m=0}^{a-2}\sum_{i=0}^mx^{(m+1)ha^2+((m+1)a+i)d}
\\&=\frac{x^{a(ha^2+ad+d)}-x^{a(ha^2+d)}}{(1-x^{ad})(1-x^{ha^2+d})}
+\frac{1-x^{a(ha^2+ad+d)}}{(1-x^{ha^2+d})(1-x^{ha^2+ad+d})}
\\&+\frac{x^{a(ha^2+ad+d)}-x^{ha^2+ad+d}}{(1-x^{ha^2+ad+d})(1-x^d)}
+\frac{x^{ha^2+ad}-x^{a^2(ha+d)}}{(1-x^{ha^2+ad})(1-x^d)}.
\end{align*}
\end{prop}

\begin{prop}
Let $A=(a, ha+(K+1)d, ha+(K+2)d, ..., ha+kd)$,  where $\gcd(a, d)=1$,  $a, h, d, K, k\in \mathbb{P}$, $K\leq \frac{k-1}{2}$ and $a\geq 2$. Let $a+K=qk+r_1$, $q\geq 0$ and $0\leq r_1\leq k-1$. Then
$$\begin{aligned}
f(x)&=\frac{x^{ha+d(K+1)}(1-x^{d(k-K)})}{1-x^d}+\frac{x^{2ha+d(k+1)}(1-x^{(ha+dk)(q-1)})(1-x^{dk})}{(1-x^{ha+dk})(1-x^d)}
\\&+\frac{x^{(q+1)ha+d(qk+1)}(1-x^{dr_1})}{1-x^d}
-\left\{
\begin{aligned}
&x^{qha+da}\ \ & \text{if} &\ \  0\leq r_1\leq K, \\
&x^{(q+1)ha+da}\ \ & \text{if} &\ \  K+1\leq r_1\leq k.\\
\end{aligned}
\right.
\end{aligned}$$
\end{prop}

\begin{prop}
Let $A=(a, ha+d, ha+2d, ha+4d, ..., ha+2kd)$,  where $(a, d)=1$,  $a, h, d, k\in \mathbb{P}$, $a>2$, $d>h$, $1\leq 2k\leq a-1$. Suppose $a-1=2k\cdot s+t, 1\leq t\leq 2k$. Then
\begin{align*}
f(x)&=1+\frac{x^{ha+d}(1-x^{(ha+2kd)(s+1)})}{1-x^{ha+2kd}}
+\frac{x^{ha+2d}(1-x^{(ha+2kd)s})(1-x^{2dk})}{(1-x^{ha+2kd})(1-x^{2d})}
\\&+\frac{x^{2ha+3d}(1-x^{(ha+2kd)s})(1-x^{2d(k-1)})}{(1-x^{ha+2kd})(1-x^{2d})}+f_1(x),
\end{align*}
where
$$\begin{aligned}
f_1(x)=
\left\{
\begin{aligned}
\frac{x^{ha(s+1)+d(2ks+2)}(1-x^{d(t-1)})}{1-x^{2d}}
+\frac{x^{ha(s+2)+d(2ks+3)}(1-x^{d(t-1)})}{1-x^{2d}}\ \ & \text{if} & t\ \text{is odd}, \\
\frac{x^{ha(s+1)+d(2ks+2)}(1-x^{dt})}{1-x^{2d}}
+\frac{x^{ha(s+2)+d(2ks+3)}(1-x^{d(t-2)})}{1-x^{2d}}\ \ & \text{if} & t\ \text{is even}. \\
\end{aligned}
\right.
\end{aligned}$$
\end{prop}

\begin{prop}
Let $A=(a, ha+d, ha+3d, ha+5d, ..., ha+(2k+1)d)$,  where $(a, d)=1$,  $a, h, d, k\in \mathbb{P}$, $a>2$, $d>h$, $1\leq 2k+1\leq a-1$. Suppose $a-1=(2k+1)s+t$, $1\leq t\leq 2k+1$. We have
\begin{align*}
f(x)&=1+\frac{x^{ha+d}(1-x^{(ha+(2k+1)d)s})(1-x^{2d(k+1)})}{(1-x^{ha+(2k+1)d})(1-x^{2d})}
+\frac{x^{2ha+2d}(1-x^{(ha+(2k+1)d)s})(1-x^{2dk})}{(1-x^{ha+(2k+1)d})(1-x^{2d})}+f_1(x),
\end{align*}
where
$$\begin{aligned}
f_1(x)=
\left\{
\begin{aligned}
\frac{x^{ha(s+1)+d((2k+1)s+1)}(1-x^{d(t+1)})}{1-x^{2d}}
+\frac{x^{ha(s+2)+d((2k+1)s+2)}(1-x^{d(t-1)})}{1-x^{2d}}\ \ & \text{if} & t\ \text{is odd}, \\
\frac{x^{ha(s+1)+d((2k+1)s+1)}(1-x^{dt})}{1-x^{2d}}
+\frac{x^{ha(s+2)+d((2k+1)s+2)}(1-x^{dt})}{1-x^{2d}}\ \ & \text{if} & t\ \text{is even}.\\
\end{aligned}
\right.
\end{aligned}$$
\end{prop}

\end{document}